\documentclass[onecolumn, 11 pt, doublespace, fullpage, letterpaper]{article}

\usepackage{amsmath}
\usepackage{geometry}                		
\usepackage{graphicx}				
\usepackage{amssymb}
\usepackage{amsmath}	
\usepackage{amssymb}	
\usepackage{amsthm}	
\usepackage{geometry}
\usepackage{graphicx}
\usepackage{cite}

\usepackage{amssymb}
\usepackage{textcomp}
\usepackage{graphicx}
\usepackage{graphics}
\usepackage{epsfig}
\usepackage{epstopdf}
\usepackage{float}
\usepackage{color}
\definecolor{mygreen}{RGB}{28,172,0} 
\definecolor{myviloet}{RGB}{170,55,241}

\usepackage{dsfont} 

\usepackage{rotating}

\usepackage{longtable}
\usepackage{contour}
\usepackage[font={small}]{caption}
\newtheorem{theorem}{Theorem}

\newtheorem{definition}{Definition}

\newtheorem{proposition}{Proposition}
\newtheorem{remark}{Remark}

\DeclareMathOperator{\sech}{sech}
 \usepackage{multirow}
 \usepackage{resizegather}
\usepackage{graphics}
\usepackage{overpic}
\usepackage{rotating}

\begin{document}

\title{Modulating Functions-Based Method for Parameters and Source Estimation in One-Dimensional Partial Differential Equations}
\author{Sharefa Asiri and Taous-Meriem Laleg-Kirati \\
{\small Computer, Electrical and mathematical Sciences and Engineering, KAUST, Thuwal, Saudi Arabia}\\
{\small sharefa.asiri@kaust.edu.sa,  taousmeriem.laleg@kaust.edu.sa}
}
\date{}
\maketitle

\begin{abstract}
In this paper, modulating functions-based method is proposed for estimating space-time dependent unknowns in one-dimensional partial differential equations. The proposed method simplified the problem into a system of algebraic equations linear in unknown parameters.  The well-posedness of  modulating functions-based solution is proven. The wave and  the fifth order KdV equations are used as examples  to show the effectiveness of the proposed method in both noise-free and noisy cases. 
\end{abstract}

\section{Introduction} \label{intro}
Inverse coefficients and  inverse source problems for partial differential equations (PDEs) are important topics in many applications such as medical imaging, seismic imaging, oil exploration, and computer tomography \cite{robinson1967predictive,CaFoSe:07,fear2000microwave,kirsch1998characterization}. 
 Various methods have been proposed to solve these problems.  The typical procedure consists in minimizing an appropriate cost function which compares measured data with the corresponding computed one. However due to ill-posedness issues, these methods often require regularization techniques such as the well-known  Tikhonov regularization \cite{MuRaCa:00}, the quasi-reversibility method  \cite{ClKl:07} and the energy regularization approach \cite{HaLiTa:11}.  The performance of the regularization techniques depends on some regularization parameters; in addition,  they  are usually  heavy computationally, especially in case of large number of parameters.  Stochastic inversion techniques such as Bayesian-based approaches are also used to estimate the parameters  of a PDE \cite{PoOlRoMa:13}.  These techniques require the knowledge of  a prior distribution of the unknown which is not always obvious. 
\par Another example of methods which have been proposed to solve inverse coefficients or inverse source problems includes some recursive approaches based on observers \cite{MoChTa:08,RaTuWe:09,AsKiZa:13}. These methods have been initially designed for state estimation of finite dimensional dynamic systems and  have been recently extended  to infinite dimensional systems. However,  they often suffer from some numerical issues generating the loss of observability, a necessary condition for the design of observer, when discretizing the PDE \cite{Zu:05}. 
 \par Responding to the growing interest in developing efficient and robust  algorithms for parameter estimation of PDEs, we propose in this paper  a method based on   the so-called modulating functions.  Modulating functions-based method has been introduced in the early fifties \cite{Sh:54, Sh:57} and has been used in parameters identification for ordinary differential equations (ODEs). In 1966, Perdreauville and Goodson \cite{PeGo:66} extended the method to the identification of  constant and space varying parameters in PDEs  using distributed measurements on a continuous-time. After that,  Fairman and Shen \cite{FaSh:70} modified the approach of Perdreauville and Goodson by using finite difference scheme to approximate the spatial derivatives. In 1997, Ungarala and B. Co  \cite{CoUn:97}  adapted the method for real-time parameters identification for ODE. Recently in 2015, the method has been combined with an optimization method to estimate fractional derivatives in fractional partial differential equations \cite{DoLiLa:15}.
Several types of  modulating functions have been proposed and used, including  sinusoidal functions \cite{Sh:57,PeGo:66}, Hermit functions \cite{Ta:68}, spline-type functions \cite{PrRi:93}, Poisson moment functionals ~\cite{SaRaRa:82}, and Hartley modulating functions \cite{PaUn:95}.
\par Modulating functions-based method has several advantages. It is computationally less costly and robust against noise. In addition, it requires neither initial nor boundary conditions. It also does not require solving the direct problem. Further, approximating the derivatives of the measurements, which are usually noisy, is avoided with this method.
\par In this paper we study the well-posedness of the modulating functions-based solution, and we investigate the effect of the number of modulating functions which, as we will show in this paper, plays a significant role. To the best of our knowledge, only in Pearson and Lee \cite{PeLe:85} a study on the number of modulating functions was considered.  They provided a guideline on how to choose this number. However, this guideline is applicable only for sinusoidal modulating functions. Moreover, it requires a priori knowledge about the system bandwidth which is mostly unknown in identification problems.
\par  The main contributions of this paper are the following. First, we extend the modulating functions-based method to estimate  space-time dependent parameters and source, separately and simultaneously, using finite number  of measurements  in both noisy and non-noisy cases.  Secondly,  the well-posedness of modulating functions-based solution is proved.  Then, a mathematical analysis of the estimation error is performed. Finally, the influence of the number of modulating functions  is investigated and discussed  independently on the choice of the modulating functions type.
\par The paper is organized  as follows. In Section \ref{Sec_main}, modulating functions-based method is presented and  the existence and  uniqueness of the modulating functions-based solution is proved. Section \ref{sec_example1} studies the source and velocity estimation in the wave equation (linear PDE) and provides some numerical simulations  in order to illustrate the effectiveness and the robustness of the proposed method. Error analysis of the noise error contribution  is also discussed.  Parameter estimation for the 5th order KdV Equation (nonlinear PDE) is studied in Section \ref{sec_example2} where some numerical simulations are depicted.  Discussion and concluding remarks are presented in Section \ref{sec_disc} and \ref{sec_conc}, respectively.

\section{Modulating Functions-Based Method}\label{Sec_main}
In this section, the problem is stated; then, the definition of modulating functions is introduced along with the  procedure of applying modulating functions-based method for estimation objectives.
\subsection{Problem Statement}\label{sec_state}   
Consider the following one-dimensional partial differential equation (1D-PDE) of order $n$ defined in the space-time domain $\Omega:=  (0,L) \times (0,T]$:
\begin{equation}\label{GENERAL1DPDE}
\mathbb{T} u(x,t) + \mathbb{P} u(x,t) =f(x,t), \quad \quad   (x,t) \in \Omega,
\end{equation}
with
\begin{equation}
\left\{
\begin{array}{ll}
Bu(x,t)=g(x,t),  \quad \quad x\in \{0,L\},\, t \in (0,T],\\
Eu(x,0)=r(x),  \quad \quad \quad x \in (0,L),
\end{array}
\right.
\end{equation}
where $x$ is the space variable, $t$ is the time variable, $L$ is the end point,  and $T$ is the final time. $B$, $E$ are  boundary and  initial conditions operators, respectively. $\mathbb{T}$ and $\mathbb{P}$ are  temporal and spatial partial differential operators such that
\begin{equation*}
\left.
\begin{array}{l}
\mathbb{T}: C^2(0,T;C^{\bar{n}}(0,L)) \rightarrow C(0,T;C^{\bar{n}}(0,L)), \qquad    \mathbb{T}u(x,t) = \sum_{r=0}^{2} a_r \partial_t^r u(x,t);\\
\mathbb{P}:C^2(0,T;C^{\bar{n}}(0,L)) \rightarrow C^2(0,T;C(0,L)), \qquad  \mathbb{P}u(x,t) = \sum_{s=1}^{\bar{n}} b_s(x,t)  \partial_x^s u(x,t); 
 \end{array}
 \right.
\end{equation*}
where  $1 \le \bar{n}\in \mathbb{N}^* \le n$, $a_r=0 \mbox{ or } 1$,  $b_s(x,t), s=1,\cdots,\bar{n}$, are the coefficients, and  $f(x,t)$ is the source term. Both the coefficients and the source are assumed to be sufficiently smooth. The above regularity requirements insure the existence of all the derivatives, in the classical sense. Depending on the considered PDE, additional conditions are required in order to ensure the existence and the uniqueness of the solution. In addition, these requirements can be slightly relaxed when the equation is formulated in the weak sense.
\par At the two end points, $0$ and $L$, the function $u(.,t)$ and its spatial derivatives up to $\bar{n}-1$ terms are supposed bounded.
\par The following problems are studied in this paper:\\
{\bf IP1: }  Estimation of the source $f(x,t)$;\\
{\bf IP2: } Estimation of the coefficients  $b_s(x,t)$;\\
{\bf IP3: } Joint estimation of the source $f(x,t)$ and the coefficients $b_s(x,t)$;\\
using  measurements of $u(x,t^*)$ and  $\mathbb{T} u(x,t^*)$ at some fixed time $t^*$. Boundary and initial conditions are not necessary to be known.
\subsection{Procedure}\label{sec_method}   
\begin{definition}
A function $\phi(x)\neq0$ is called a modulating function of order $l$ ($l \in \mathbb{N^*}$) if it  satisfies:
\begin{equation}\label{prop_modu}
\left\{
\begin{array}{ll}
\phi(x) \in C^{l} ([0,L])& (a)\\
\mbox{and}&\\
 \phi^{(p)}(0)=\phi^{(p)}(L)=0 , \quad \forall p=1,2,\cdots,l-1,& (b)
 \end{array}
 \right.
\end{equation}
where $L >0$ and $p$ refers to the order of the derivative.
\end{definition}
The basic steps for solving the three inverse problems: IP1, IP2, and IP3, where $u(x,t^*) $ and $\mathbb{T} u(x,t^*)$ are the measured data, are presented in the following:\\
{\bf STEP 1:} Fix the time in equation  (\ref{GENERAL1DPDE}) at $t^*$, and then multiply the equation by the modulating function $\phi(x)$:
\begin{equation}\label{STEP1}
 \mathbb{T} u(x,t^*) \phi(x) +  \mathbb{P} u(x,t^*)  \phi(x) =  f(x,t^*) \phi(x).
\end{equation}
{\bf STEP 2:} Integrate over the space interval:
\begin{equation}\label{STEP2}
\int_0^L \mathbb{T} u(x,t^*) \phi(x)\, \mathrm{d}x+ \int_0^L \mathbb{P} u(x,t^*)  \phi(x) \, \mathrm{d}x= \int_0^L f(x,t^*) \phi(x) \, \mathrm{d}x.
\end{equation}
It is worth noting that in case of noisy measurements, which is usually the case in practice, the integral in this step has an effect to dampen and filter  the noise.\\
{\bf STEP 3:}   Apply the integration by parts formula to the second integral in (\ref{STEP2}):
\begin{equation}\label{STEP3}
\int_0^L \mathbb{T} u(x,t^*) \phi(x) \, \mathrm{d}x+  \int_0^L  u(x,t^*)  \mathbb{Q} \phi(x) \ \mathrm{d}x=
\int_0^L f(x,t^*) \phi(x) \ \mathrm{d}x.
\end{equation}
where $\mathbb{Q} \phi(x) =  \sum_{s=1}^{\bar{n}}  (-1)^{s} \partial_x^s \left[ b_s(x,t^*) \phi(x) \right]$. This step transfers all the spatial derivatives of the solution to derivatives of the modulating function, which is usually known analytically  (\ref{prop_modu}.a). Also, the boundary conditions, that appear in the integration by parts process,  are eliminated thanks to the second property of modulating functions (\ref{prop_modu}.b).
\par By solving (\ref{STEP3}), one can obtain the unknown coefficients, source, or both. It is worth noting that modulating functions-based method does not require solving a direct problem that may be computationally very complex and especially for high order equations as it is required with standard optimization methods. 
\par  As the unknown is identified at fixed time $t^*$, one can also identify it  at some other fixed times; then interpolate the obtained data to find the unknown in the whole domain $\Omega$.
\subsection{Properties of Modulating Functions-Based Solution}\label{exist_unique}
Now we can discuss the well-posedness of the modulating functions-based solution. First, since all the functions in (\ref{STEP3}) are sufficiently smooth, and the product of two smooth functions has finite integral,
all the integrals in (\ref{STEP3}) converge for each $t^* \in (0,T]$. Hence, modulating functions-based solution of each inverse problem exists. The uniqueness and stability of the modulating functions-based solution for the three inverse problems are guaranteed by the next theorems.
\begin{theorem}[Uniqueness of modulating functions-based solution]
Assume that the measurements $u(x,t^*)$ and $ \mathbb{T} u(x,t^*)$ exist and are sufficiently smooth. Then there exist a unique solution for {\bf IP1}, {\bf IP2}, and {\bf IP3}  satisfying (\ref{STEP3}).
\end{theorem}
\begin{proof}
{\bf IP1:} Assume that $f$,$\bar{f}$ $\in C([0,L])$ satisfy
\begin{equation}\label{source_uniq}
\int_0^L \mathbb{T} u(x,t^*) \phi(x) \, \mathrm{d}x+  \int_0^L  u(x,t^*)  \mathbb{Q} \phi(x) \ \mathrm{d}x=
\int_0^L f(x,t^*) \phi(x) \ \mathrm{d}x = \int_0^L \bar{f}(x,t^*) \phi(x) \ \mathrm{d}x
\end{equation}
for all modulating function $\phi(x)$. Then
\begin{equation}
\int_0^L \left[ f(x,t^*) - \bar{f}(x,t^*) \right] \phi(x) \ \mathrm{d}x = 0
\end{equation}
for all  $\phi(x)$, where $\phi(x)$ satisfies (\ref{prop_modu}); hence $f-\bar{f}=0$\\
{\bf IP2:}  For simplicity and without lose of generality, let $\bar{n}=1$. In the following, we prove here the uniqueness of estimating $b_1(x,t^*)$.  As in {\bf IP1}, assume that $b_1(x,t^*)$, $\bar{b}_1(x,t^*)$ satisfy
\begin{equation}\label{coeff_uniq}
\begin{split}
\int_0^L \mathbb{T} u(x,t^*) \phi(x) \, \mathrm{d}x- \int_0^L f(x,t^*) \phi(x) \ \mathrm{d}x \\
 =  \int_0^L   \left[ b_1^\prime(x,t^*) \phi(x) + b_1(x,t^*) \phi^\prime(x) \right] u(x,t^*)\ \mathrm{d}x\\
 =   \int_0^L   \left[ \bar{b}_1^\prime(x,t^*) \phi(x) + \bar{b}_1(x,t^*) \phi^\prime(x) \right] u(x,t^*) \ \mathrm{d}x
\end{split}
\end{equation}
for any modulating function $\phi(x)$. Hence, 
\begin{equation}
 \int_0^L   \left[b_1(x,t^*) - \bar{b}_1(x,t^*) \right] u(x,t^*) \phi^\prime(x) \ \mathrm{d}x +   \int_0^L  \left[b_1^\prime(x,t^*) - {\bar{b}}^\prime_1(x,t^*) \right] u(x,t^*) \phi(x)  \ \mathrm{d}x=0,
\end{equation}
for all $\phi(x)$ and $\phi^\prime(x)$. Thus, 
\begin{equation}
\left[b_1(x,t^*) - \bar{b}_1(x,t^*) \right] u(x,t^*)=0;
\end{equation}
and
\begin{equation}
\left[b_1^\prime(x,t^*) - {\bar{b}}^\prime_1(x,t^*) \right] u(x,t^*)=0.
\end{equation}
Since $u(x,t^*)$ is the measurement, $u(x,t^*) \neq 0$; therefore, $b_1=~\bar{b}_1$.\\
{\bf IP3:} Here we prove the uniqueness of estimating $f(x,t^*)$ and $b_1(x,t^*)$ jointly.
 Let  $f$,$\bar{f}$ and $b_1$, $\bar{b}_1$  satisfy 
\begin{equation}\label{joint_uniq}
\begin{split}
 \int_0^L \mathbb{T} u(x,t^*) \phi(x) \, \mathrm{d}x \qquad \qquad \qquad \qquad\qquad\qquad\qquad\qquad \qquad\qquad \\
= \int_0^L f(x,t^*) \phi(x) \ \mathrm{d}x +   \int_0^L  \left[ b_1^\prime(x,t^*) \phi(x) + b_1(x,t^*) \phi^\prime(x) \right]  u(x,t^*) \ \mathrm{d}x  \\
=\int_0^L \bar{f}(x,t^*) \phi(x) \ \mathrm{d}x+ \int_0^L  \left[ \bar{b}_1^\prime(x,t^*) \phi(x) + \bar{b}_1(x,t^*) \phi^\prime(x) \right] u(x,t^*) \ \mathrm{d}x
\end{split}
\end{equation}
for all modulating function $\phi(x)$. Then
\begin{equation}\label{joint_uniq_0}
\begin{split}
\int_0^L \left[ f(x,t^*) - \bar{f}(x,t^*) \right] \phi(x) \ \mathrm{d}x  \qquad\qquad\qquad \qquad\qquad\qquad \qquad\qquad\qquad\\
+ \int_0^L   \left\{  \left[b_1(x,t^*) - \bar{b}_1(x,t^*) \right] \phi^\prime(x) + \left[b_1^\prime(x,t^*) - {\bar{b}}^\prime_1(x,t^*) \right] \phi(x) \right\} u(x,t^*) \ \mathrm{d}x=0
\end{split}
\end{equation}
for all $\phi(x)$ and at $u(x,t^*) \neq 0$. Therefore, (\ref{joint_uniq_0}) is true if and only if 
\[f(x,t^*) - \bar{f}(x,t^*)=0\] and 
\[\left[b_1(x,t^*) - \bar{b}_1(x,t^*) \right] \phi^\prime(x) + \left[b_1^\prime(x,t^*) - {\bar{b}}^\prime_1(x,t^*) \right] \phi(x) = 0\]
for all $\phi(x)$. Consequently,  $f(x,t^*) = \bar{f}(x,t^*)$ and $b_1(x,t^*) = \bar{b}_1(x,t^*)$.
\end{proof}
\begin{theorem}[Stability of modulating functions-based solution]
Modulating functions-based solution of {\bf IP1}, {\bf IP2}, and {\bf IP3} is stable in the sense that: {\bf IP1}: if $f$ and $\tilde{f}$ are the solutions of (\ref{STEP3}) with respect to the data $\{u,\mathbb{T}u\}$ and $\{\tilde{u},\widetilde{\mathbb{T}u}\}$, respectively; such that
 \begin{equation}\label{stability_noisy_data}
\left.
\begin{array}{cccc}
\tilde{u}=u+\eta_1, &  &  &  \widetilde{\mathbb{T}u}=\mathbb{T}u + \eta_2;
\end{array}
\right.
 \end{equation}
where $\eta_1, \eta_2$ are the noise functions such that $\eta_1 \in H^{\bar{n}}[0,L]$ and $\eta_2 \in L^{2}[0,L]$. Then  $\|f-\tilde{f}\|_{L^2} \longrightarrow 0$ as $\|\eta_1\|_{H^{\bar{n}}},\|\eta_2\|_{L^2}\longrightarrow 0$. Similarly for {\bf IP2} and {\bf IP3}.
\end{theorem}
\begin{proof}
 For simplicity and without lose of generality, let $\bar{n}=1$.
 \par \noindent {\bf IP1}: Since $f$ and $\tilde{f}$ are solutions to (\ref{STEP3}) w.r.t (\ref{stability_noisy_data}), one can write
 \begin{equation}
 \left.
\begin{array}{ll}
\left| \displaystyle{\int_0^L} \left[\tilde{f}-f \right] \phi  \ \mathrm{d}x \right| &= \left|  \displaystyle{\int_0^L}  \eta_2 \phi  \ \mathrm{d}x +  \displaystyle{\int_0^L}  \eta_1 \left( - b_1 \phi \right)^\prime  \ \mathrm{d}x  \right|, \\
 & \le \left|   \displaystyle{\int_0^L} \eta_2 \phi  \ \mathrm{d}x  \right|+ \left|  \displaystyle{\int_0^L}  \eta_1 \left( - b_1 \phi \right)^\prime  \ \mathrm{d}x  \right|. \label{stab_after_tring}
 \end{array}
\right.
\end{equation}
By applying Holder inequality to the right hand side of (\ref{stab_after_tring}), we end up with
 \begin{equation}
\left.
\begin{array}{ll}
\left| \displaystyle{\int_0^L} \left[\tilde{f}-f \right] \phi  \ \mathrm{d}x \right| & \le \left\| \eta_2 \right\|_{L^2} \left\| \phi \right\|_{L^2} + \left\| \eta_1 \right\|_{L^2}  \left\| b_1^\prime \phi+ b_1 \phi^\prime  \right\|_{L^2},\\
 & \le \displaystyle{\left\| \eta_2 \right\|_{L^2} \left\| \phi \right\|_{L^2} + \left\| \eta_1 \right\|_{L^2}  \left\{  \left\| b_1^\prime\right\|_{L^2} \left\| \phi \right\|_{L^2}+ \left\|b_1\right\|_{L^2} \left\| \phi^\prime  \right\|_{L^2} \right\} },\\
& \le\displaystyle{ \left\| \eta_2 \right\|_{L^2} \left\| \phi \right\|_{L^2} + \left\| \eta_1 \right\|_{H^1} \left\| b_1\right\|_{H^1}  \left\{ \left\| \phi \right\|_{L^2}+  \left\| \phi^\prime  \right\|_{L^2} \right\}; } \label{stability_ineq1_IP1}
 \end{array}
\right.
\end{equation}
where the prime symbol here refers to the space derivatives. On the other hand,
\begin{equation}\label{stability_ineq2_IP1}
\left| \int_0^L \left[\tilde{f}-f \right] \phi  \ \mathrm{d}x \right| \le \left\| \tilde{f}-f \right\|_{L^2} \left\| \phi \right\|_{L^2}.
\end{equation}
Moreover, from the PDE (\ref{GENERAL1DPDE}): if $f_1$ and $f_2$ are two source terms, then
\begin{eqnarray}
f_1= \mathbb{T}{u}_1 + \mathbb{P}u_1, \\
f_2= \mathbb{T}{u}_2 +\mathbb{P} u_2.
\end{eqnarray}
The difference, in a norm,  between the two equations is:
\begin{equation}
\begin{array}{ll}
\|f_1-f_2\|_{L^2} &= \|(\mathbb{T}{u}_1 -\mathbb{T}{u}_2) + (\mathbb{P} u_1-\mathbb{P}u_2) \|_{L^2}\\
& \le \|\mathbb{T}{u}_1 -\mathbb{T}{u}_2\|_{L^2}\ + \| \mathbb{P} (u_1-u_2) \|_{L^2}
\end{array}
\end{equation}
In the case $\bar{n}=1$, $\mathbb{P}u=b_1u^\prime$; and if  $u_1-u_2=\eta_1$, $\mathbb{T}{u}_1 -\mathbb{T}{u}_2=\eta_2$, and $f_1$, $f_2$ represent the estimated and the exact source, then 
\begin{equation}
\begin{array}{ll}
\| \tilde{f}-f \|_{L^2} & \le \|\eta_2\|_{L^2} + \|b_1 \eta_1^\prime\|_{L^2}\\
& \le \|\eta_2\|_{L^2} + \|b_1\|_{H^1} \|\eta_1\|_{H^1} \label{19}
\end{array}
\end{equation}
Thus, by multiplying (\ref{19}) with $\|\phi\|_{L^2}$, we obtain:
\begin{equation}\label{f8}
\| \tilde{f}-f \|_{L^2}  \|\phi\|_{L^2} \le \|\eta_2\|_{L^2} \|\phi\|_{L^2} +  \|b_1\|_{H^1} \|\eta_1\|_{H^1} \|\phi\|_{L^2}.
\end{equation}
Combining (\ref{stability_ineq1_IP1}), (\ref{stability_ineq2_IP1}), and (\ref{f8}) gives us
\begin{equation}\label{stability_ineq3_IP1}
\left\| \tilde{f}-f \right\|_{L^2} \le  \left\| \eta_2 \right\|_{L^2} + \left\| \eta_1 \right\|_{H^1}  \left\| {b_1} \right\|_{H^1} \left[ 1+ \dfrac{\| \phi^\prime \|_{L^2}} {\| \phi\|_{L^2}} \right]. 
\end{equation}
The nice properties of the modulating functions guaranty that the ratio $ \dfrac{\| \phi^\prime \|_{L^2}} {\| \phi\|_{L^2}}$ in (\ref{stability_ineq3_IP1}) is small enough, and  under the assumption that  $\left\| {b_1} \right\|_{H^1}$ is small enough, then as $\|\eta_1\|_{H^1}, \|\eta_2\|_{L^2} \longrightarrow 0$, $ \| \tilde{f}-f \|_{L^2} \longrightarrow 0$; which proves the stability, in term of norm, for modulating functions-based solution of {\bf IP1}. \par Similarly, one can prove the stability for modulating functions-based solutions of {\bf IP2} and  {\bf IP3}.
\end{proof}
 As a main feature of modulating functions-based method, estimating the unknown in equation (\ref{STEP3}) can be simplified into solving a system of linear algebraic equations. The details are presented in the next proposition.
\begin{proposition}\label{prop_main} 
Let $ \sum_{i=1}^I \gamma_i \xi_i(x)$ be a basis expansion of the unknown, $f(x,t)$ or $b_s(x,t)$,  at a fixed time $t^*$, where  $\xi_i(x)$  and $\gamma_i$, for $i=1,\cdots,I$, are basis functions and their coefficients, respectively; and let $\{\phi_m(x)\}_{m=1}^{m=M}$ be a class of modulating functions of order $\bar{n}^*$ ($\bar{n}^* \ge \bar{n}$) with $M \ge I$. Then the unknown coefficients can be estimated by solving:
\begin{equation}\label{STEP3_m}
\int_0^L \mathbb{T} u(x,t^*) \phi_m(x) \, \mathrm{d}x+  \int_0^L  u(x,t^*)  \mathbb{Q} \phi_m(x) \ \mathrm{d}x=
\int_0^L f(x,t^*) \phi_m(x) \ \mathrm{d}x, \quad m=1,\cdots,M,
\end{equation}
which can be written into a system of the form:
\begin{equation}\label{STEP3_LAE}
A {\Gamma} = {Y},
\end{equation}
where the elements of matrix $A$ are combination between the basis and the modulating functions. The elements $y_m$ of the vector ${Y} \in \mathbb{R}^M$ are:\\
{\bf IP1}: $y_m= \displaystyle{\int_0^L} \mathbb{T} u(x,t^*) \phi(x) +  u(x,t^*)  \mathbb{Q} \phi(x) \ \mathrm{d}x$;\\
{\bf IP2}: $y_m=\displaystyle{\int_0^L}  \left[ \mathbb{T} u(x,t^*) - f(x,t^*)\right] \phi(x) \ \mathrm{d}x$; \\
{\bf IP3}: $y_m= \displaystyle{\int_0^L} \mathbb{T} u(x,t^*) \phi(x) \, \mathrm{d}x$.\\
$\Gamma$ is a vector of the unknown basis coefficients $\gamma_i,$ $i=1,2,\cdots,I$. 
\end{proposition}
\begin{proof}
By applying the steps used to derive (\ref{STEP3}) but w.r.t $\phi_m(x)$, one can obtain (\ref{STEP3_m}) and therefore obtain (\ref{STEP3_LAE}).
\end{proof}
%
%
%
\begin{remark}\label{rem_1}
 The structure of the matrix $\mathcal{A}$ along with the nice regularity properties of the modulating functions confirm that 
 $\mathcal{A}$ has always full-column rank; hence, system (\ref{system_ct_compound}) has always a unique solution; and therefore, modulating functions-based solution is a well-posed problem in the strict mathematical sense.  However, it can be numerically  ill-conditioned and therefore exhibit numerical instability which is mainly, in our method, related to the nature and the number of the basis and the modulating functions; therefore, these factors should be chosen appropriately.
\end{remark}
%
\section{Source and velocity estimation for the wave equation}\label{sec_example1} 
\subsection{Method} 
Consider the following one-dimensional wave equation in the domain $\Omega~:= (0,L) \times (0,T]$:
\begin{equation}
\left\{
\begin{array}{llll}\label{wave equation}
 u_{tt}(x,t) - c(x,t) u_{xx}(x,t)=f(x,t), & & & (x,t) \in \Omega \\
 u(0,t)=g_1(t), \ \  u(L,t)=g_2(t), &   &  & t\in [0,T]\\
 u(x,0)=r_1(x),\ \  u_t(x,0)=r_2(x),&  & & x\in (0,L)
\end{array}
\right.
\end{equation}
where $g_i(t)$, $r_i(x)$, $i=1,2$, are respectively the boundary and the initial conditions which are assumed not to be  necessarily known.  The source function is denoted by $f(x,t)$. $c(x,t)$ is the square of the velocity of wave propagation at  point $x$ and time $t$, and it is assumed to be positive and bounded. All the functions are assumed to be sufficiently regular.
\par The goal of this example is to illustrate how modulating functions-based method can be applied to solve IP1, IP2, and IP3. Here, the measurements are the displacement $u(x,t^*)$ and the acceleration $u_{tt}(x,t^*)$ at a fixed time instant. Next propositions study these different problems.
\begin{proposition}\label{IP12}
Let $ \sum_{i=1}^I \gamma_i \xi_i(x)$ be a basis expansion of the unknown, source (IP1) or velocity (IP2),  in (\ref{wave equation}) at a fixed time $t^*$, where  $\xi_i(x)$  and $\gamma_i$, for $i=1,\cdots,I$,  are basis functions and basis coefficients, respectively. Let $\{\phi_m(x)\}_{m=1}^{m=M}$ be a class of at least second order modulating functions with $M \ge I$. Then, the unknown coefficients $\gamma_i$, $i=1,2,\cdots,I$, can be estimated by solving the system: 
\begin{equation}\label{system_ct_compound}
 {\bf \mathcal{A}} { \Gamma} = {K},
\end{equation}
where the components of the $M \times I$ matrix $ {\bf \mathcal{A}}$ have the form:
\begin{eqnarray}\label{system_ct_compound_compoC}
\left.
\begin{array}{ll} 
\mbox{IP1: } &\mathcal{A}_{mi} =  \displaystyle{\int_0^L} \phi_m(x) \xi_i(x) \, \mathrm{d}x,\\
\mbox{IP2: } &\mathcal{A}_{mi}=  \displaystyle{\int_0^L} u(x,t^*) \left[ \xi_i''(x) \phi_m (x)+\right.\\
& \qquad \qquad  2 \xi_i' (x)  \phi'_m(x)+ \left. \xi_i (x) \phi_m'' (x)\right] \, \mathrm{d}x,
\end{array}
\right.
\end{eqnarray}
for $m=1,\cdots,M$ and $i=1,\cdots,I$; the components of the vector ${K} \in \mathbb{R}^M$ are:
 \begin{eqnarray}\label{system_ct_compound_compoK}
 \left.
\begin{array}{ll} 
\mbox{IP1: } & K_m= \displaystyle{\int_0^L} \phi_m(x) u_{tt}(x,t^*) \mathrm{d}x - c \int_0^L \phi^{\prime\prime}_m(x) u(x,t^*)\, \mathrm{d}x,\\
\mbox{IP2: }  & K_m=  \displaystyle{\int_0^L} [u_{tt}(x,t^*) - f(x,t^*) ] \phi_m(x)\, \mathrm{d}x,
\end{array}
\right.
\end{eqnarray}
and ${ \Gamma}$ is the vector of the unknowns $\gamma_i$, $i=1,\cdots,I$.
\end{proposition}
\begin{proof} The proof follows the steps described previously in Subsection \ref{sec_method}; see Appendix~\ref{app_IP12}.
\end{proof}
%
\begin{remark}\label{rem_2}
 Proposition \ref{IP12} can be directly applied to estimate space-dependent velocity $c(x)$ or space-time-dependent velocity $c(x,t)$. For constant velocity case, the method is simpler and more accurate since there is no need for approximating the constant $c$ using basis expansion or any other approximation. In this case, the estimated velocity is: 
\begin{equation}
\hat{c} =\frac{ {\bf \mathcal{A}}^t{K}}{ {\bf  \mathcal{A}}^t{\bf  \mathcal{A}}}.
\end{equation}
where the components of the vectors ${\bf  \mathcal{A}}$ and ${K}$ have the form
\begin{equation}
\mathcal{A}_m=\int_0^L \phi_{m}^{\prime\prime}(x) u(x,t^*) \, \mathrm{d}x, \quad m=1,\cdots,M,
\end{equation}
and 
\begin{equation}
K_m= \int_0^L \phi_m(x) \left[ u_{tt}(x,t^*)- f(x,t^*) \right] \, \mathrm{d}x,  \quad m=1,\cdots,M,
\end{equation}
respectively.
\end{remark}
In the next proposition, the joint estimation problem (IP3) is studied.
\begin{proposition}\label{joint}
Let $f_I(x,t^*) = \sum_{i=1}^{I} \gamma_i \xi_i(x)$ and  $c_J(x,t^*) = \sum_{j=1}^{J} \beta_j \upsilon_j(x)$ be  basis expansion for the unknown source and  velocity, respectively; where  $\xi_i(x)$, $\upsilon_j(x)$  and $\gamma_i$, $\beta_j$, for $i=1,\cdots,I$ and $j=1,\cdots,J$, are basis functions and basis coefficients, respectively. Let $\{\phi_m(x)\}_{m=1}^{m=M}$ be a class of modulating functions with $M \ge I+J$ and $l \ge 2$. Then, the unknown coefficients $\gamma_i$ and $\beta_j$ can be estimated by solving the system: 
\begin{equation}\label{system_joint}
 {\bf \mathcal{B}} { \digamma} = {Q},
\end{equation}
where  ${\bf \mathcal{B}}$ is  $M \times (I+J)$ matrix which can be written as ${\bf \mathcal{B}}=[\Xi \quad \Upsilon]$ such that $\Xi$ is $M \times I$ matrix with components:
\begin{equation}\label{F_joint}
\Xi_{mi} =  \int_0^L \phi_m(x) \xi_i(x) \, \mathrm{d}x, \qquad i=1,\cdots,I; \, m=1,\cdots,M,
\end{equation}
and $\Upsilon$ is $M \times J$ matrix with components:
\begin{equation}\label{C_joint}
\Upsilon_{mj} =   \int_0^L u(x,t^*)  \left[ \upsilon_j^{\prime\prime}(x) \phi_m(x)+ 2 \upsilon_j^{'}(x) \phi_m ^{'}(x)+  \upsilon_j(x) \phi^{\prime\prime}_m(x) \right]\, \mathrm{d}x,
\end{equation}
for $m=1,\cdots,M$ and $j=1,\cdots,J$. The components of the vector ${Q}$ are:
 \begin{equation}\label{K_joint}
 Q_m=  \int_0^L  u_{tt}(x,t^*) \phi_m(x) \, \mathrm{d}x,
\end{equation}
and  the vector of the unknowns is 
\begin{equation}\label{Gamma_joint}
{ \digamma}=
\left[
\begin{array}{cccccccc}
\gamma_1 &\gamma_2&\cdots&\gamma_I&\beta_1 &\beta_2&\cdots&\beta_J 
\end{array}
\right]^{tr}.
\end{equation}
Here $tr$ refers to transpose. $u(x,t^*)$ and $u_{tt}(x,t^*)$ are the measurements.
\end{proposition}
\begin{proof} Appendix \ref{app_joint} provides the steps of the proof.
\end{proof}
\subsection{Error Analysis}\label{sec_error} 
It is known that measured data are noisy. Therefore, it would be interesting to study the effect of noise on the modulating functions-based method.  We have shown in the previous subsection that the studied inverse problems are simplified into solving a system of linear equations, say $\mathcal{A}\Gamma=K$. 
 Here, we  are going to study the noise error contribution in the load vector $K$ of IP1. Similarly, one can analyze the error in IP2 and IP3.
\par Let  $h_1(x)$ and $h_2(x)$ be bounded noises, i.e. $\|h_1(x)\|_\infty= \delta_1$ and  $\|h_2(x)\|_\infty= \delta_2$, such that:
\begin{equation}\label{noisy_data}
\left\{
\begin{array}{l}
\hat{u} (x,t^*) = u(x,t^*) + h_1(x),  \\
\hat{u}_{tt} (x,t^*) = u_{tt}(x,t^*) + h_2(x); 
\end{array}
\right.
\end{equation}
If $\underline{e}$ represents the noise error contribution vector of IP1, then from (\ref{system_ct_compound_compoK})
\begin{equation}\label{err_vector}
\underline{e}=
\left[
\begin{array}{cccc}
  e_1& e_2   & \cdots & e_M 
\end{array}
\right]^{tr};
\end{equation}
where 
\begin{equation}\label{err_compo}
e_m = \hat{K}_m -K_m =  \int_0^L \left[ \phi_m(x) h_2(x) - c  \phi^{\prime\prime}_m(x) h_1(x) \right] \mathrm{d}x.
\end{equation}
Clearly from (\ref{err_compo}), this error depends not only on the noise but also on the functions  $ \phi_m(x)$ and $ \phi^{\prime\prime}_m(x)$. However, $\phi_m(x)$  is a continuous function in the closed bounded interval $[0,L]$; thus, by the boundedness theorem, $\phi_m(x)$ is bounded. That is, there exists a real number $\nu_1$ such that
\begin{equation}
|\phi_m(x)| \le \nu_1 \qquad \forall x \in [0,L].
\end{equation}
Similarly for $\phi^{\prime\prime}_m(x)$: 
\begin{equation}
|\phi^{\prime\prime}_m(x)| \le \nu_2 \qquad \forall x \in [0,L];
\end{equation}
where $\nu_2\in \mathbb{R}$. In the next proposition, we give an error bound for this noise error contribution vector.
\begin{proposition}
Let  $\nu~=~\max\{\nu_1,\nu_2\}$ and $\delta~=~\max\{\delta_1,\delta_2\}$; then
\begin{equation}\label{bound_error}
\|\underline{e} \|_1 \le M  \nu \delta (1 + c) L,
\end{equation}
where $\underline{e}$ is given in (\ref{err_vector}).
\end{proposition}
\begin{proof}
The 1-norm  of (\ref{err_compo}) is:
\begin{equation}
\|\underline{e} \|_1 = \sum_{m=1}^M \left| \int_0^L \phi_m(x) h_2(x) - c \phi^{\prime\prime}_m(x) h_1(x) \, \mathrm{d}x \right|;
\end{equation}
which can be bounded by:
\begin{equation}
\|\underline{e} \|_1 \le  \sum_{m=1}^M \int_0^L  \left| \phi_m(x) h_2(x) - c \phi^{\prime\prime}_m(x) h_1(x)\right|  \, \mathrm{d}x.
\end{equation}
By applying the triangle inequality, one can obtain:
\begin{equation}
\|\underline{e} \|_1 \le \sum_{m=1}^M \int_0^L  \left| \phi_m(x) \right| \left| h_2(x) \right| + c \left|\phi^{\prime\prime}_m(x)\right| \left| h_1(x)\right|  \, \mathrm{d}x.
\end{equation}
From the boundedness of the noise and the modulating functions, we can arrive to:
\begin{equation}
\|\underline{e} \|_1 \le  \delta \sum_{m=1}^M \int_0^L  \nu (1 + c)\, \mathrm{d}x.
\end{equation}
Hence,
\begin{equation}
\|\underline{e} \|_1\le \delta M  \nu (1 + c) L.
\end{equation}
\end{proof}
The inequality in (\ref{bound_error}) proves that the noise error contribution upper bound is affected by the noise in the data ($\delta$), the length of the interval ($L$), and the type of the used modulating functions which is reflected in $\nu$. It also suggests the use of modulating functions which have less maximum norm derivative as well. In addition,  a fixed upper bound for the noise error contribution can be achieved by finding a trade-off between the number of modulating functions $M$ and the level of noise $\delta$.
\par In order to show the efficiency of the presented method, IP1, IP2, and IP3 are simulated numerically in the next subsection.
\subsection{Numerical Simulations}\label{sec_num_wave} 
To perform numerical simulations, system (\ref{wave equation}) has been first discretized using finite differences scheme. Then a set of synthetic data  has been generated using the following parameters: $L=3$, $T=1$, and $N_x=N_t=3001$, where $N_x$ and $N_t$ are space and time grid sizes, respectively. The method has been implemented in Matlab and applied in noise-free and noisy cases. In the noise-corrupted case, $1\%$, $3\%$, $5\%$, and $10\%$ white Gaussian random noises with zero means have been added to the data where the noise level is evaluated  using  $\frac{\|u_{exact}-u_{approximate}\|_2}{\|u_{exact}\|_2} \times ~100$; see Figure~\ref{measurements}.
\begin{figure} 
\centering
\begin{tabular}{cc}
 \epsfig{file= 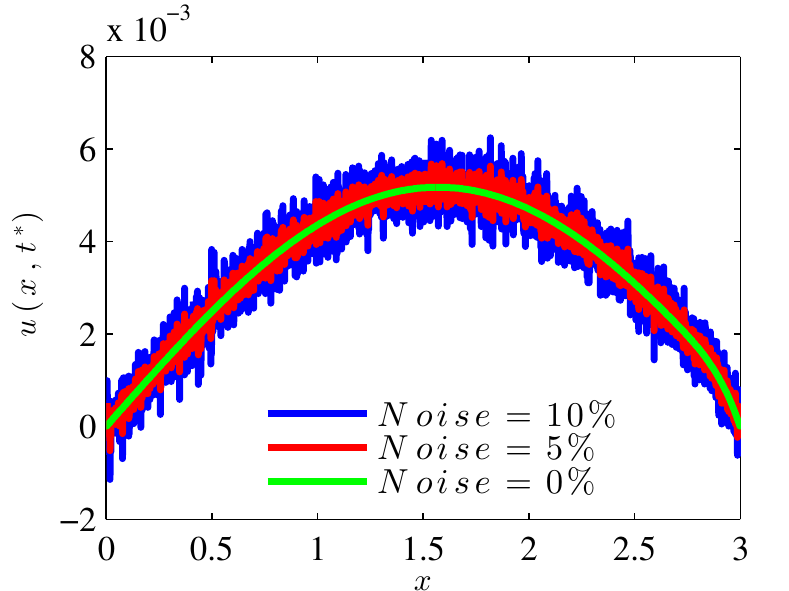,height=0.4\linewidth,width=0.48\linewidth,clip=} & \epsfig{file= 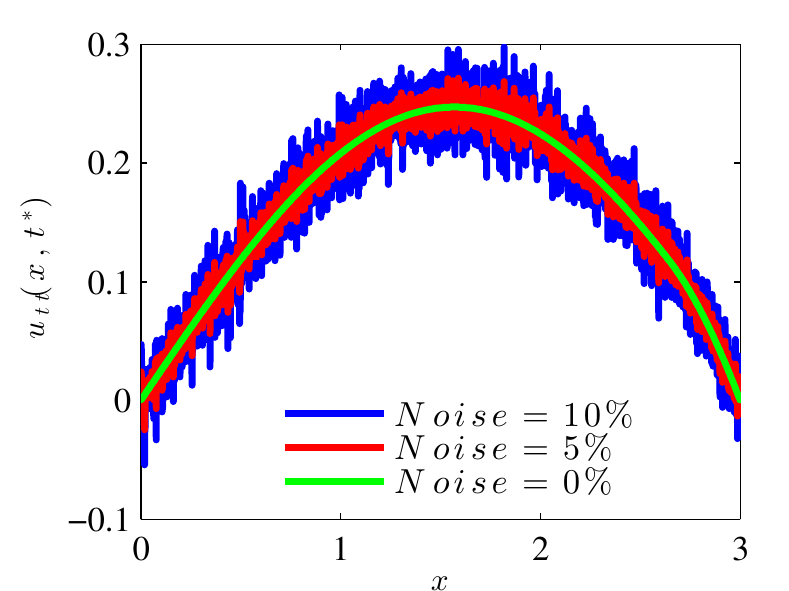,height=0.4\linewidth,width=0.5\linewidth,clip=} \\
 {\bf a} & {\bf b}
\end{tabular}
\caption{Green solid lines represent the exact measurements while the red and blue lines represent the noisy measurements with $5\%$ and $10\%$ of noise, respectively. The sub-figure {\bf a} is for the displacement $u(x,t^*)$ while {\bf b} for  the acceleration $u_{tt}(x,t^*)$; both at fixed time $t^*$.}
\label{measurements}
\end{figure}
Regarding the choice of the modulating functions, as mentioned before there are many functions that satisfy (\ref{prop_modu}), see e.g. ~\cite{Ta:68,PrRi:93,SaRaRa:82,PaUn:95,PeLe:85}. In this paper, polynomial-type modulating functions \cite{LiLaGiPe:13} have been used for their simplicity. They have the following form:
\begin{equation}\label{poly_modu}
\phi_m(x) = (L-x)^{q+m} x^{q+M+1-m},
\end{equation}
where $m=1,2,\cdots,M$, $M$ is the number of modulating functions, and $q \in \mathbb{R}^+$ is a degree of freedom which should be chosen such that $\phi_m(x)$ is of at least second order for all $m$;  see Figure  \ref{polynomial_modulating_functions_ind11_noise5}. For the basis functions $\xi_i(x)$, polynomial basis have been used.
 \begin{figure}
\centering
\epsfig{file=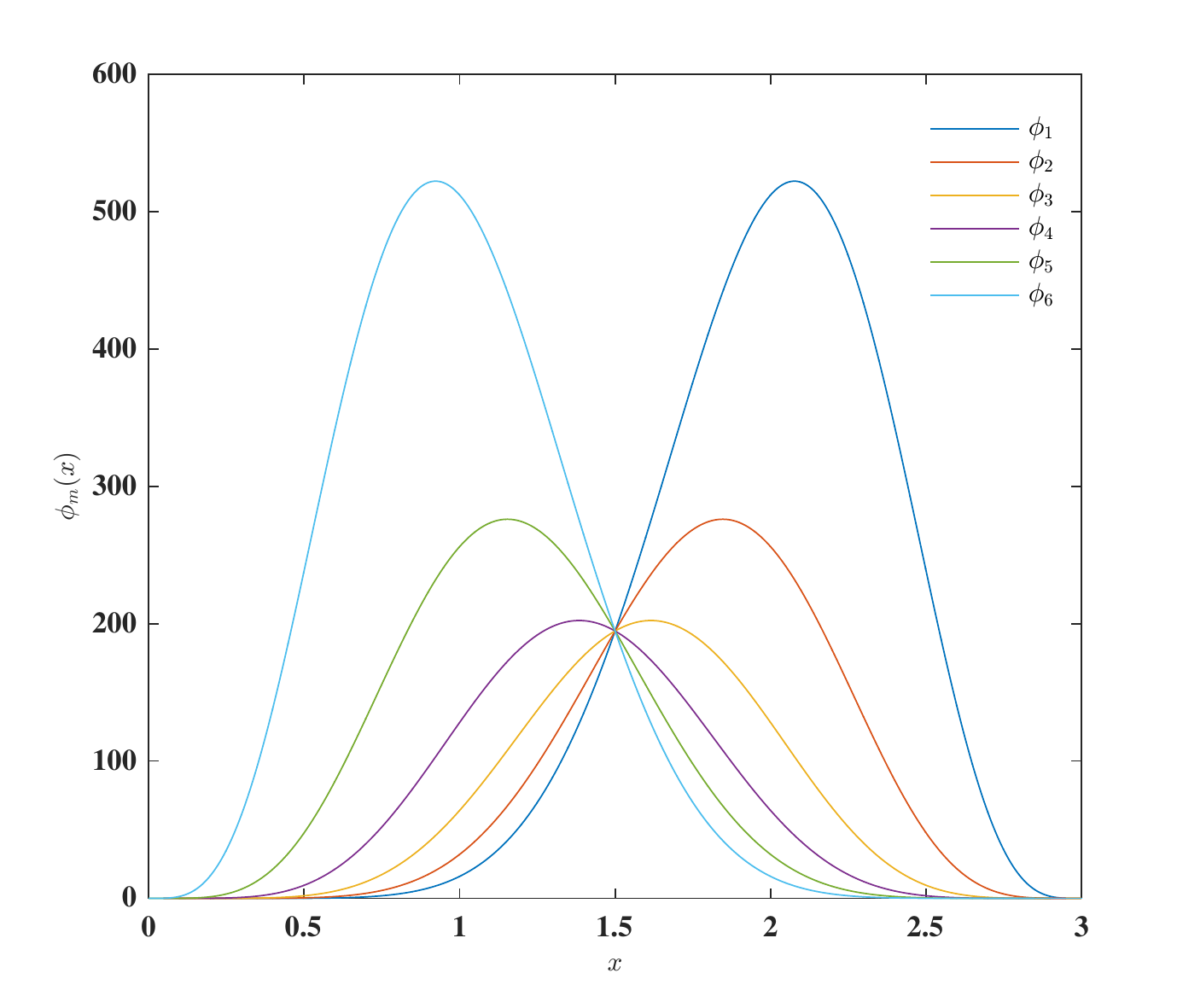,height=0.4\linewidth, width=0.52\linewidth,clip=} 
\caption{Six polynomial modulating functions where $M=6$ and $q$=3. }
\label{polynomial_modulating_functions_ind11_noise5}
\end{figure}
\subsubsection{{\rm IP 1}}
The exact source has been chosen to be $f(x,t)=\sin(x)t^2$, and $c=0.5$. At a fixed time, the exact source and the estimated one versus different noise levels are presented in Figure~\ref{elsevier_f_all}, and the corresponding relative errors are shown in Table~\ref{Relative errors of Fx}.
\begin{figure} 
\centering
\begin{tabular}{cc}
\epsfig{file=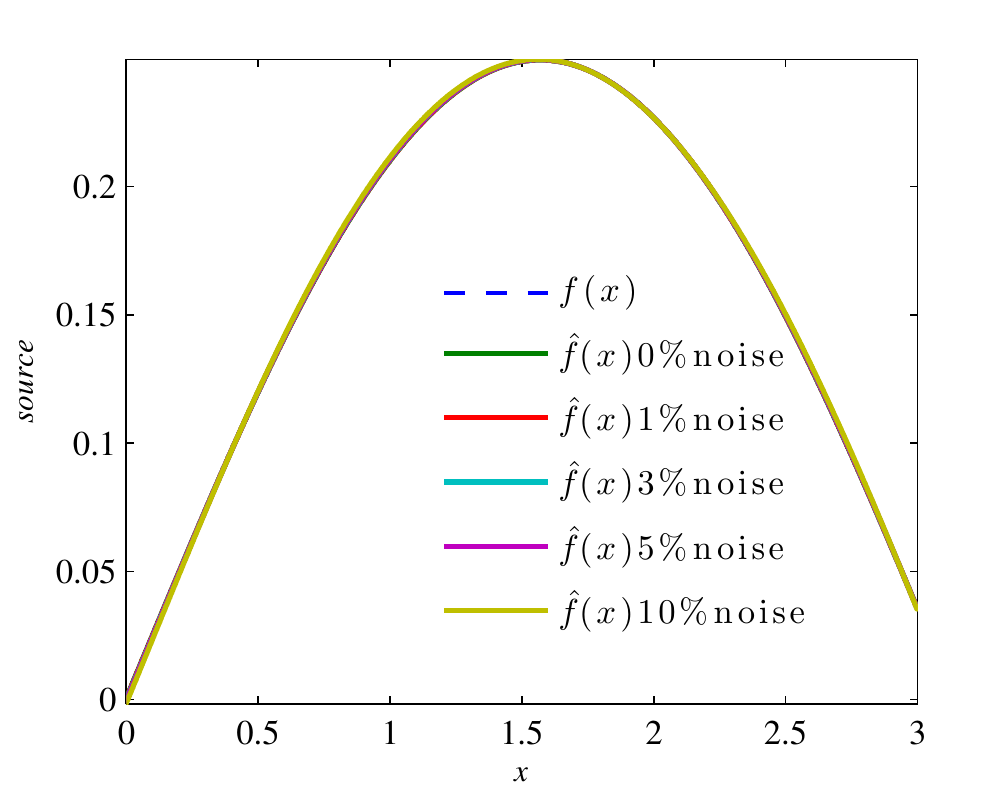,height=0.42\linewidth,width=0.5\linewidth,clip=} & \epsfig{file= 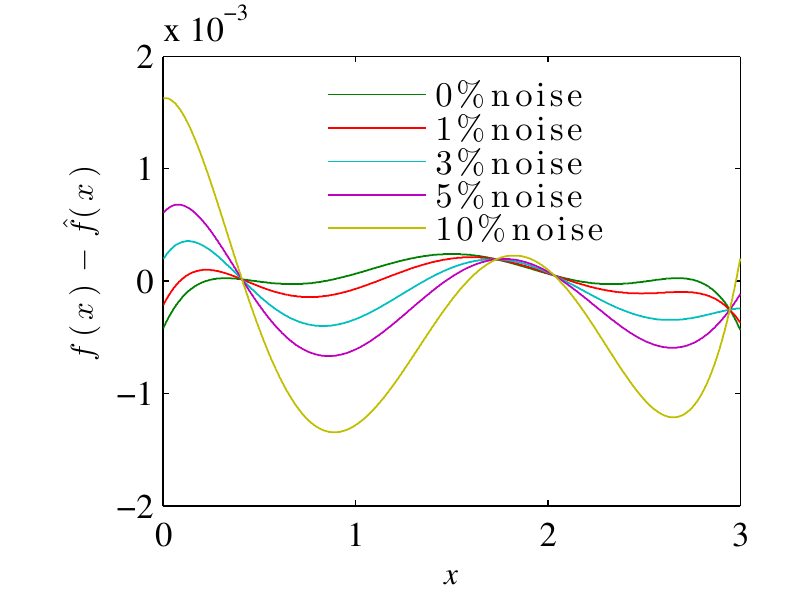, height=0.43\linewidth,width=0.5\linewidth,clip=}\\
 {\bf a} & {\bf b}
\end{tabular}
\caption{{\bf a}: Exact source (blue dashed) and estimated ones (colored solid) w.r.t different noise levels: $0\%,1\%,3\%,5\%,$ and $10\%$ of noise; where $q=3$ and  $M=27$. {\bf b}:~Estimation errors for the results in {\bf a}.}
\label{elsevier_f_all}
\end{figure}
\begin{table}
\caption{Relative errors of $\hat{f}(x)$ versus different noise levels.}
\label{Relative errors of Fx}
\begin{center}
\begin{tabular}{cc} \hline
Noise Level &  Relative Error \\ \hline
0\%	&	    0.0728	\%	 \\
1\%	&	    0.0695	\%	\\
3\%	&	    0.1365	\%	\\
5\%	&	    0.2283	\%	\\
10\%	&	    0.4711	\%	\\ \hline
\end{tabular}
\end{center}
\end{table}
 Figure~\ref{elsevier_f_noise5_3positions} exhibits the exact source and the estimated one at three different times $t_1^*$, $t_2^*$ and $t_3^*$. These results, which are obtained at only three time instants,  are interpolated to estimate the source $f(x,t)$ as shown in Figure~\ref{elsevier_f_interp_noise5}. For the interpolation, there are different methods that can be used to interpolate (estimate) new data points from known ones, see, for example \cite{Ra:96}. Here, polynomial interpolation method has been used for illustration. The relative errors for estimating $f(x,t)$ against different noise levels is presented in Table~\ref{Relative errors of F}. As we can see, even with $10\%$ of noise, the estimation error is only $0.4711	\%$.
\begin{figure} 
\centering
\begin{tabular}{ccc}
\epsfig{file= 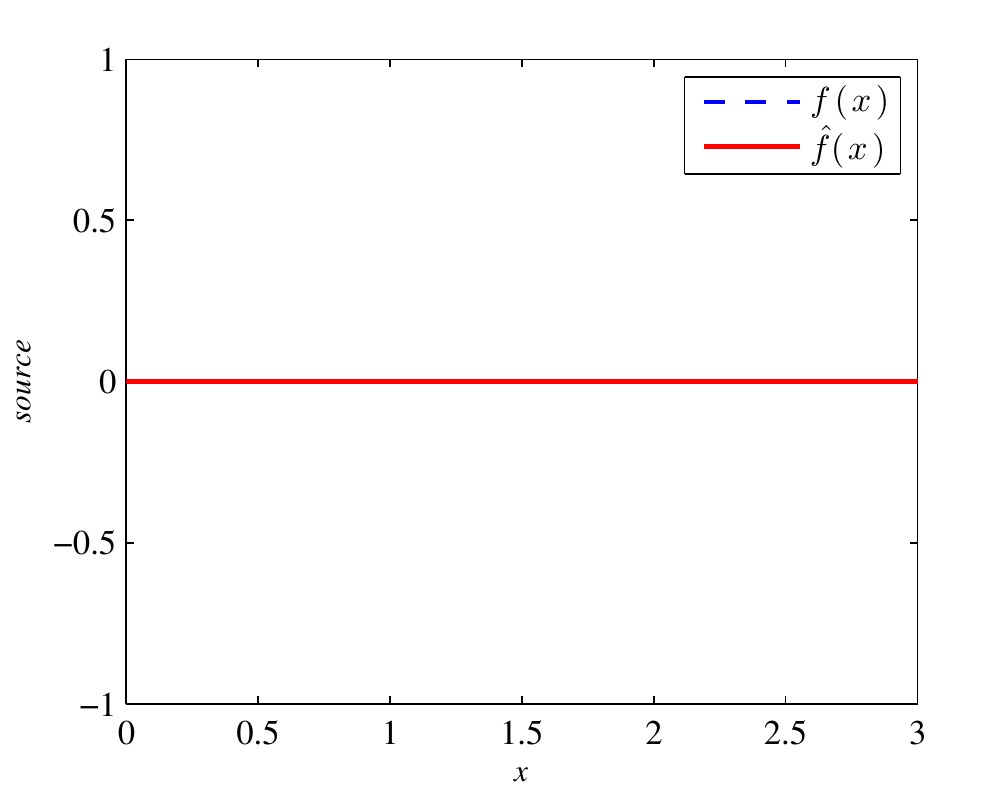,height=0.28\linewidth,width=0.3\linewidth,clip=} & \epsfig{file= 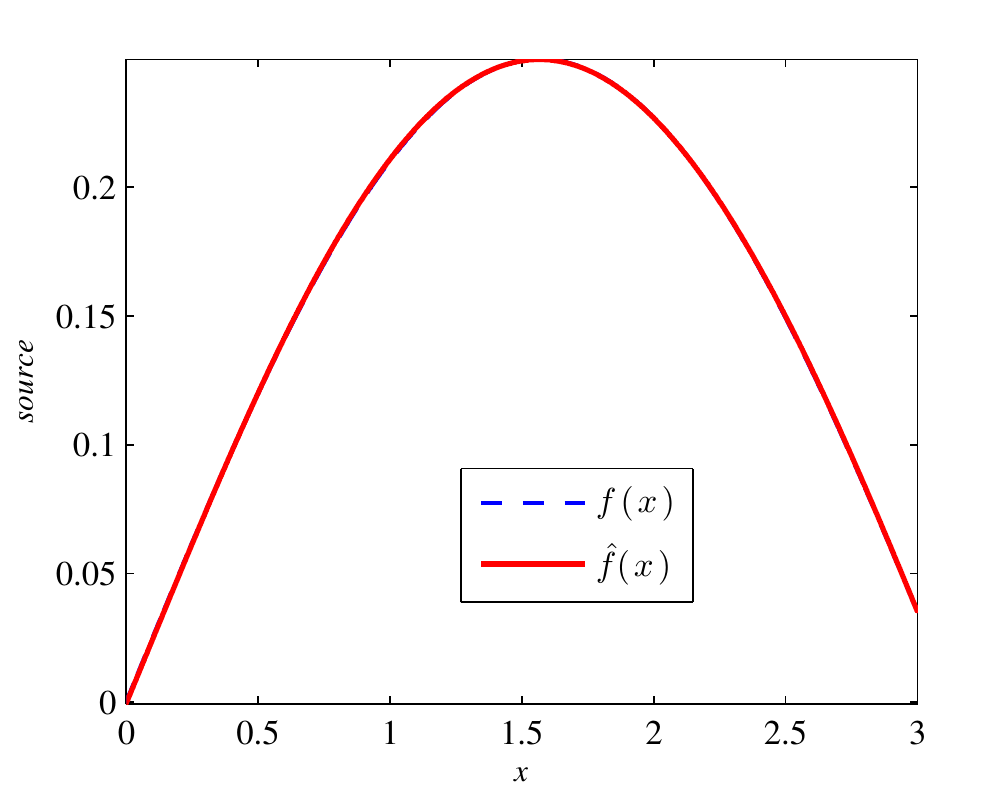,height=0.28\linewidth,width=0.3\linewidth,clip=} & \epsfig{file= 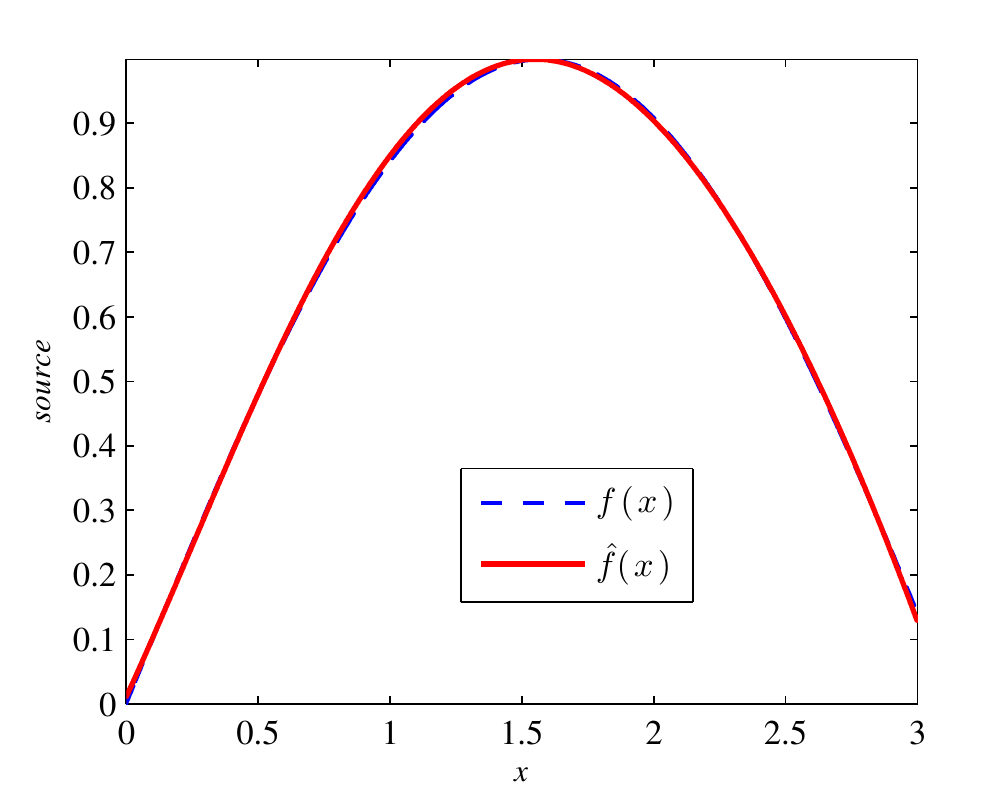,height=0.28\linewidth,width=0.3\linewidth,clip=} \\
$f(x,t^*_1)$ &$f(x,t^*_2)$ & $f(x,t^*_3)$
\end{tabular}
\caption{Exact source (blue dashed) and estimated one (red solid) at three fixed times. Noise level $=5\%, q=3$ and  $M=27$.}
\label{elsevier_f_noise5_3positions}
\end{figure}
\begin{figure} 
\centering
\begin{tabular}{cc}
 \epsfig{file= 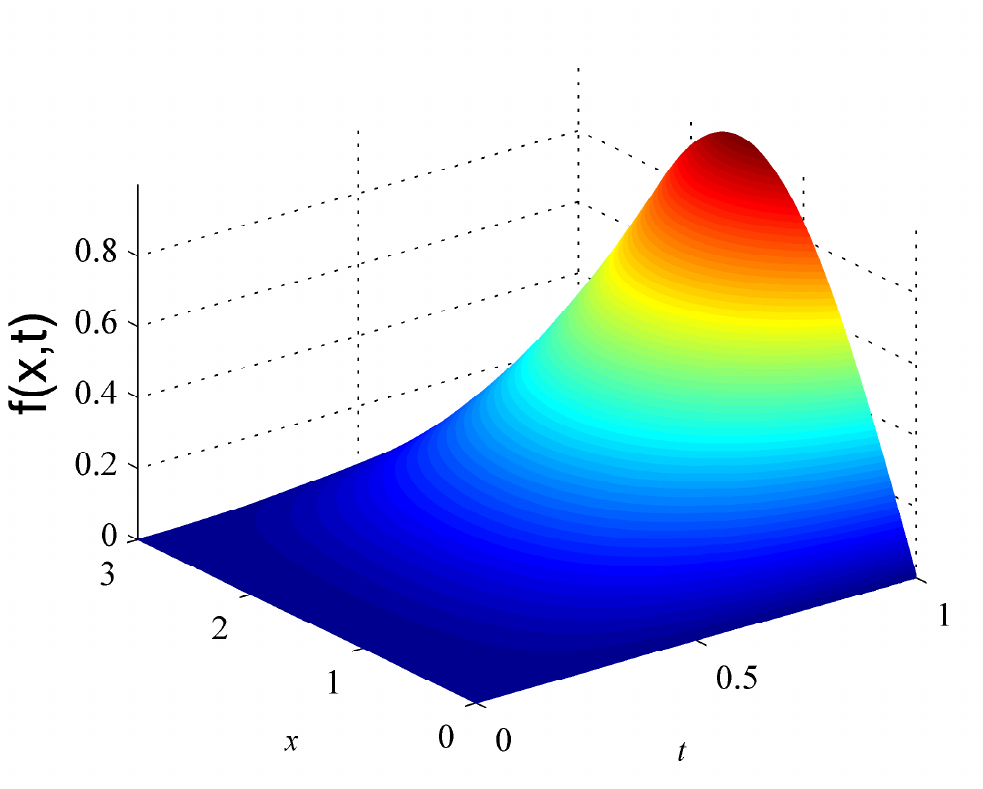,width=0.5\linewidth,clip=} & \epsfig{file= 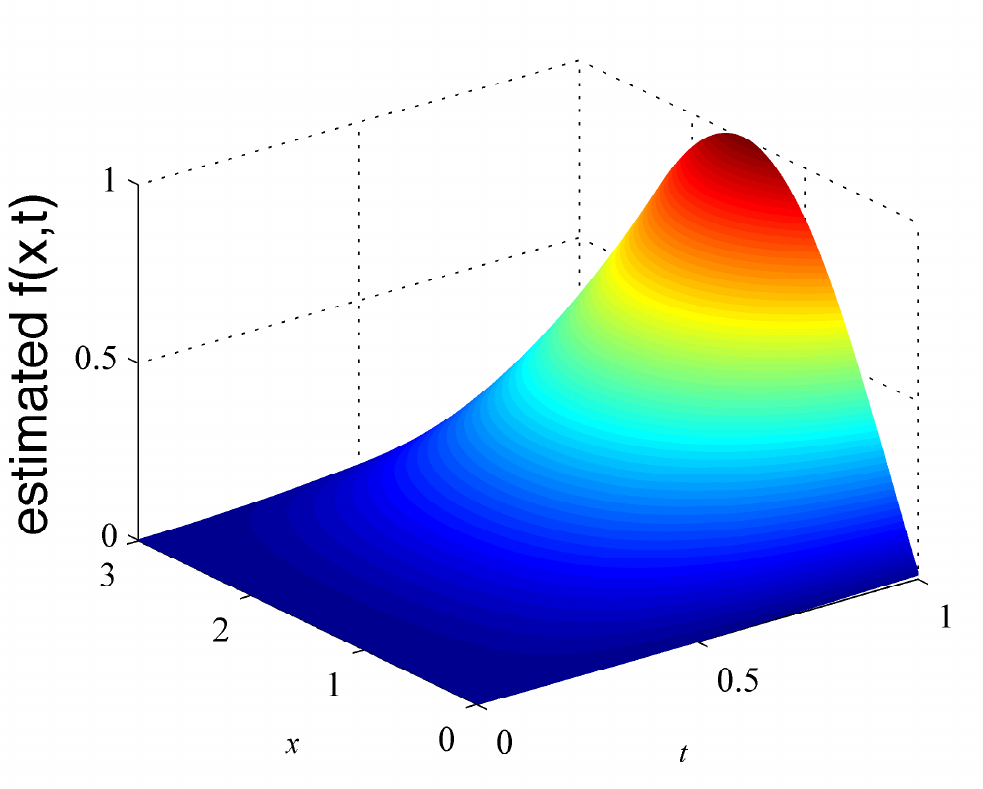,width=0.5\linewidth,clip=} \\
 {\bf a} & {\bf b}
\end{tabular}
\caption{{\bf a}: Exact source $f(x,t)=sin(x)t^2$; {\bf b}: Estimated source after interpolating the data in Figure~\ref{elsevier_f_noise5_3positions}. Noise level $=5\%, q=3$ and  $M=27$.}
\label{elsevier_f_interp_noise5}
\end{figure}
\begin{table} 
\caption{Relative errors of $\hat{f}(x,t)$ versus different noise levels.}
\label{Relative errors of F}
\begin{center}
\begin{tabular}{cc} \hline
Noise Level &  Relative Error \\ \hline
0\%	&	  0.07284	\%	 \\
1\%	&	    0.11102	\%	\\
3\%	&	    0.31809	\%	\\
5\%	&	    0.53732	\%	\\
10\%	&	    1.0905	\%	\\ \hline
\end{tabular}
\end{center}
\end{table}
\par Figure~\ref{bound_error_numer} illustrates the effect of different parameters namely the length of the interval $L$, the number of modulating functions $M$, and the level of noise on the noise error contribution. As we can see, these results confirm the ones we obtained in equation (\ref{bound_error}) in the error analysis subsection.
\begin{figure} 
\centering
\begin{tabular}{ccc}
 \epsfig{file= 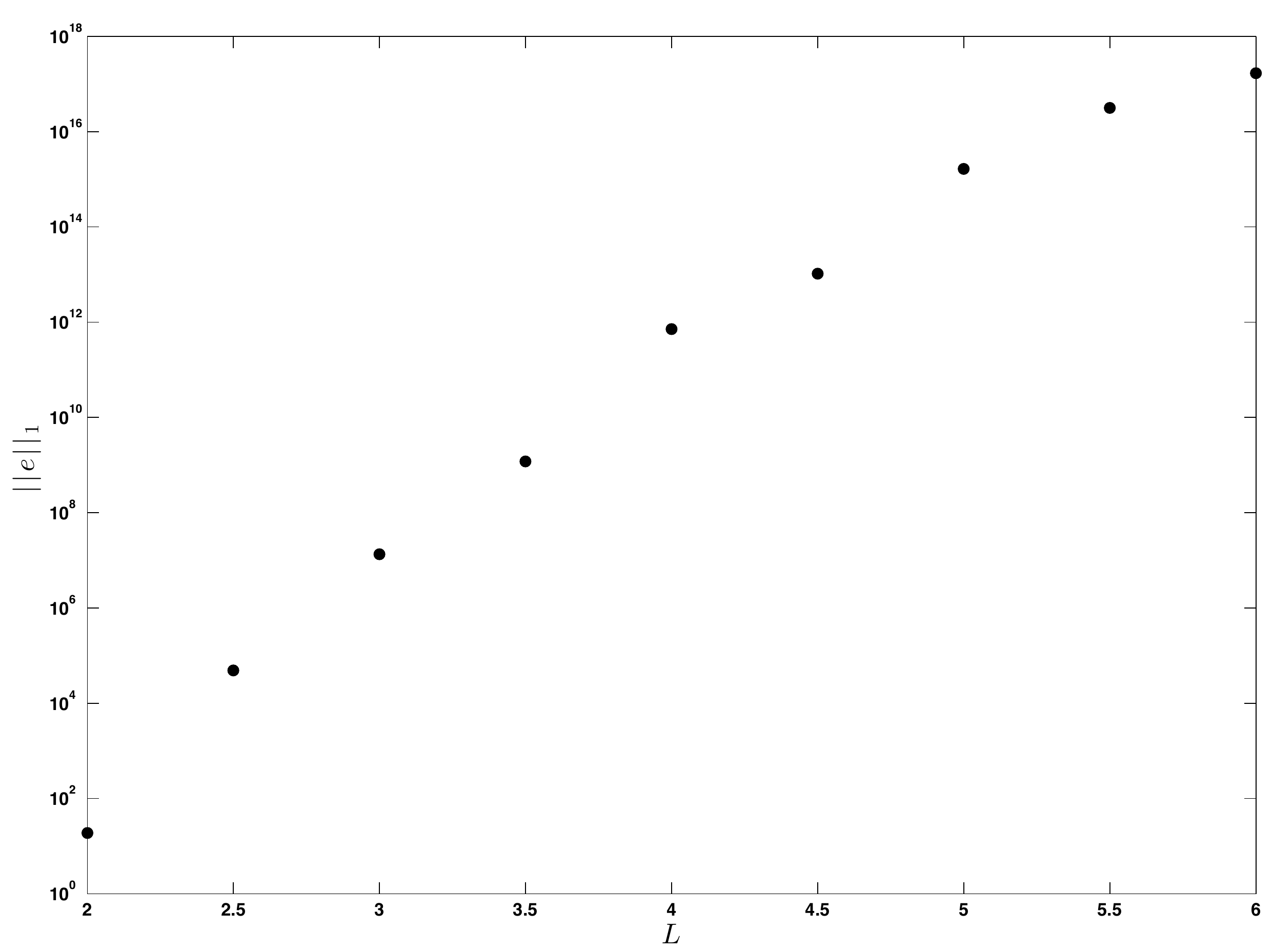,width=0.45\linewidth,clip=}  & \epsfig{file= 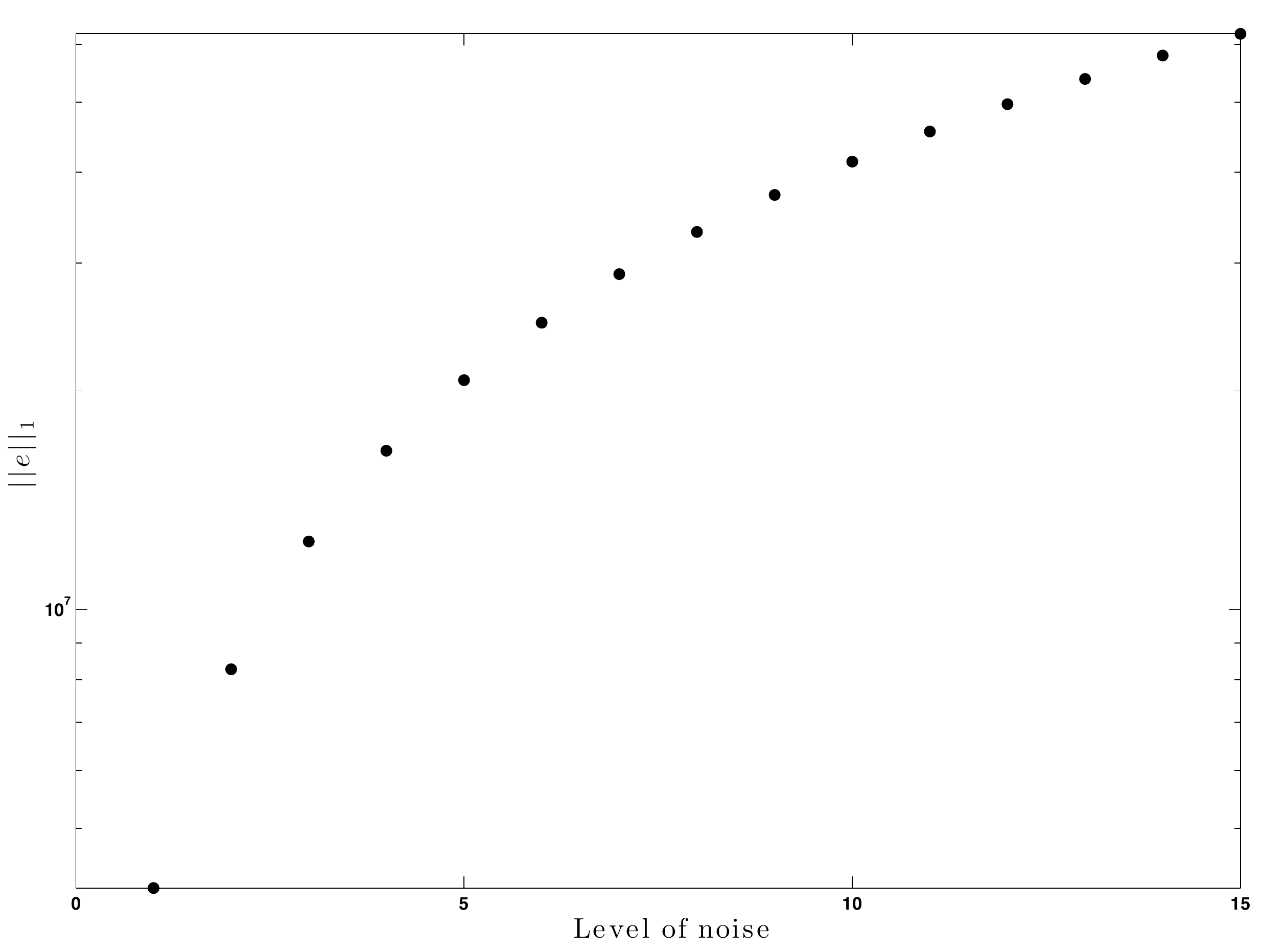,width=0.45\linewidth,clip=} \\
  {\bf a} & {\bf b} \\
  \epsfig{file= 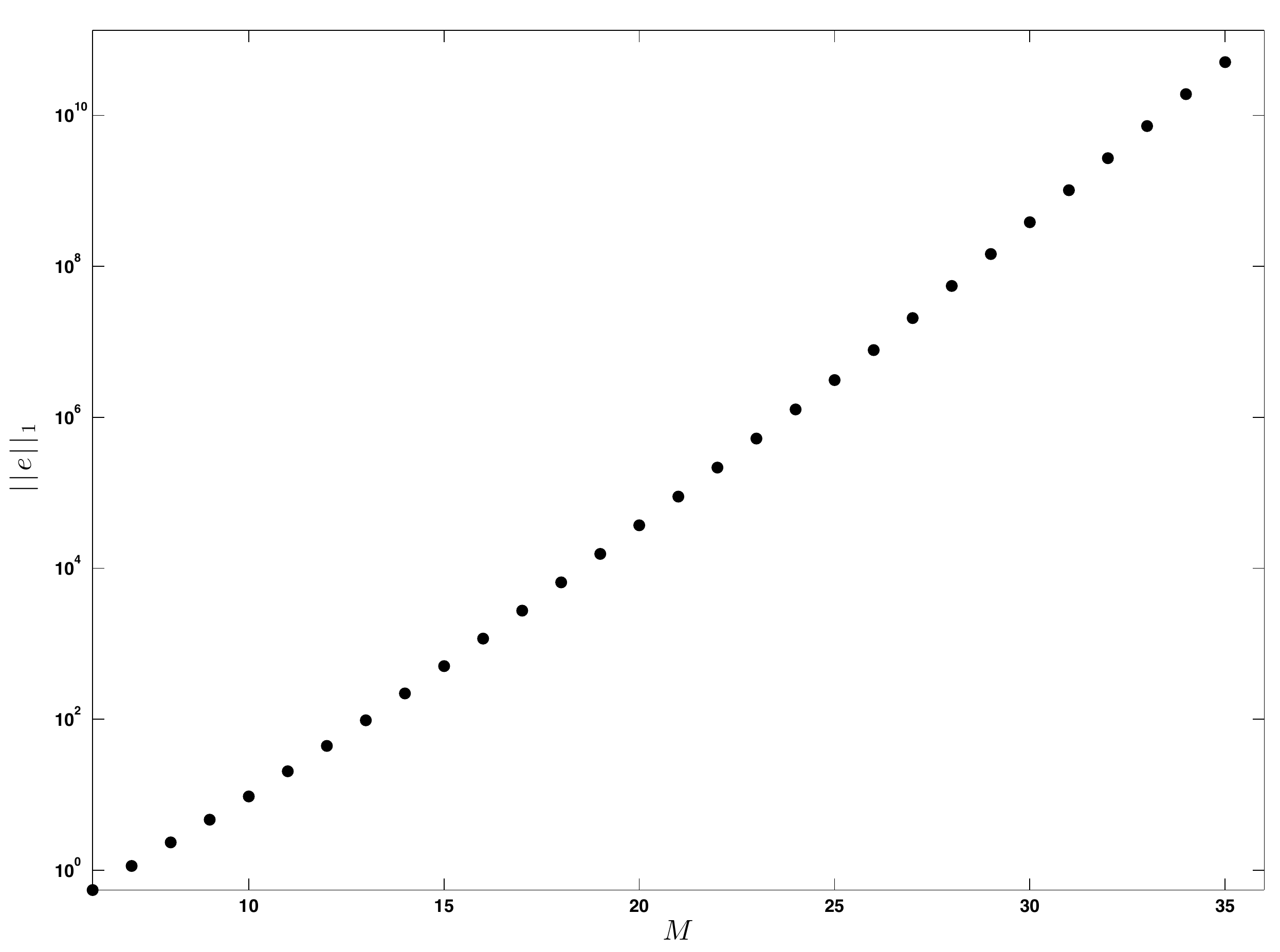,width=0.45\linewidth,clip=} & \epsfig{file= 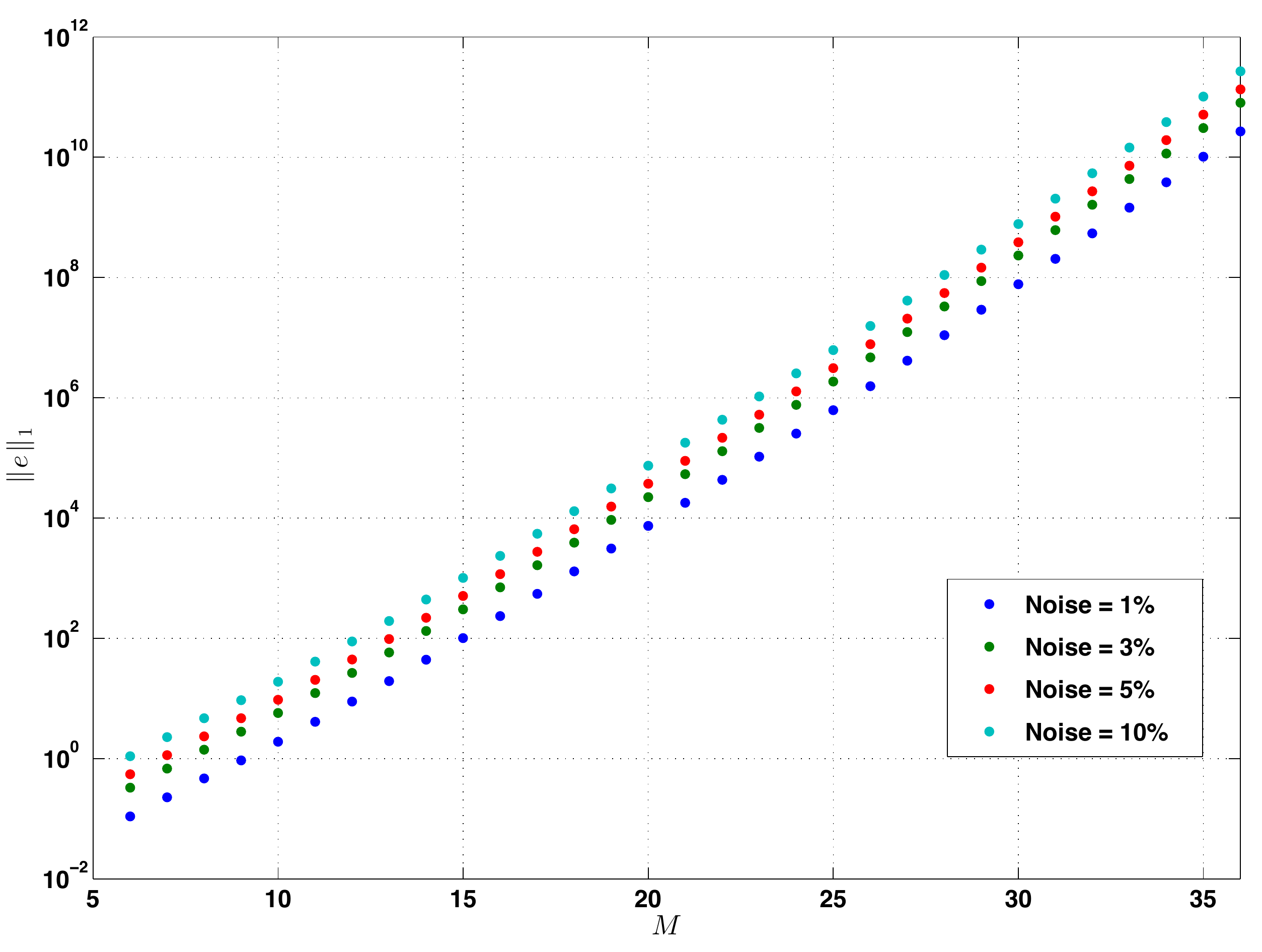,width=0.45\linewidth,clip=}\\
 {\bf c} & {\bf d}
\end{tabular}
\caption{Noise error contribution versus: {\bf a:} The length of the space interval, $L$; {\bf b:} The level of noise;  {\bf c:} The number of modulating functions, $M$; {\bf d:} The number of modulating functions and with different noise levels.}
\label{bound_error_numer}
\end{figure}
\subsubsection*{{\rm IP 1} with measurements interpolation:} In this part, we assume the availability of the measurements at only few points; i.e. we have $u(x_i,t^*)$ and $u_{tt}(x_i,t^*)$ at some points $x_i$, $i=1,2,\cdots,\bar{N}_x$. In this case, the available measurements can be interpolated in order to approximate the measurements over the whole domain $\Omega=[0,L]$. The relative errors corresponding to different noise levels are shown in Table~\ref{source_measurements_interpolation}, where $\bar{N}_x= 15$ points ($\bar{N}_x = 7 \% N_x$) and spline interpolation have been used.
\begin{table}
\caption{Relative errors of $\hat{f}(x)$ versus different noise levels when only $7\%$ of the measurements are used (measurements interpolation)}
\label{source_measurements_interpolation}
\begin{center}
\begin{tabular}{cc} \hline
Noise Level &  Relative Error \\ \hline				
0\%	&	    0.2225	\%	 \\
1\%	&	    0.2144	\%	\\
3\%	&	    0.2962	\%	\\
5\%	&	    0.4411	\%	\\
10\%	&	    0.8642	\%	\\ \hline
\end{tabular}
\end{center}
\end{table}
\subsubsection{{\rm IP 2:}}
Three cases have been tested: constant velocity, space varying velocity and space-time dependent velocity.  The exact values of the velocity in the three cases are chosen as follows: $c=0.5$, $c(x)=x^2$ and $c(x,t)=(xt)^2$, respectively. Table~\ref{elsevier_c_all} shows the estimated values and the relative errors of the constant velocity coefficient with different noise levels. The results of estimating $c(x)$ and $c(x,t)$ are presented in Figure~\ref{elsevier_cX_all}, Figure~\ref{elsevier_cXT_noise5_3positions}, and Figure~\ref{elsevier_cXT_interp_noise5}. Their corresponding relative errors are shown in Table~\ref{error_cX} and Table~\ref{error_cXT}.
\par Figure~\ref{elsevier_cX_all_Hermite} and Table~\ref{elsevier_cX_all_Hermite} show the results of the space varying velocity $c(x)=x^2$ when Hermite polynomials are used as basis functions.
\begin{table}
\caption{The estimated results for constant velocity coefficient with different noise levels where the exact is $c=0.5$.}
\label{elsevier_c_all}
\begin{center}
\begin{tabular}{ccc} \hline
Noise Level &  Estimated Value & Relative Error (\%) \\ \hline
0\%	&	          0.49999583341812		&	   8.3332e-04	\%	 \\
1\%	&	         0.500195743015861		&	   3.9149e-02	\%	\\
3\%	&	         0.500590887585075		&	   1.1818e-01	\%	\\
5\%	&	         0.500979919835552		&	   1.9598e-01	\%	\\
10\%	&	         0.501926672127135		&	   3.8533e-01	\%	\\ \hline
\end{tabular}
\end{center}
\end{table}
 \begin{figure} 
\centering
\begin{tabular}{cc}
\epsfig{file=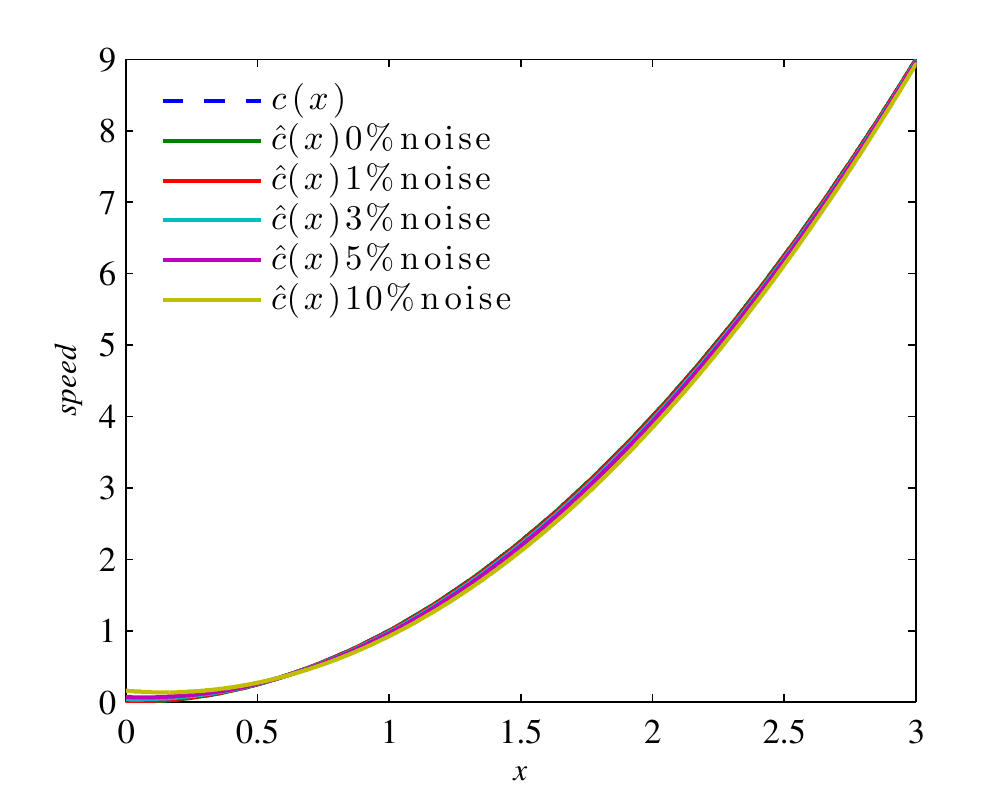,height=0.42\linewidth,width=0.5\linewidth,clip=} & \epsfig{file= 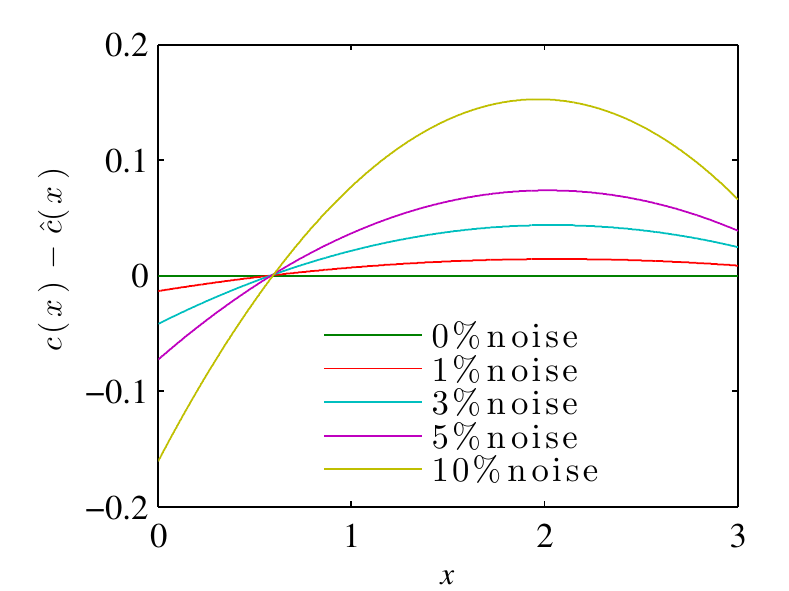, height=0.43\linewidth,width=0.5\linewidth,clip=}\\
 {\bf a} & {\bf b}
\end{tabular}
\caption{{\bf a}: Exact  space-dependent velocity $c(x)$ (blue dashed) and  estimated ones $\hat{c}(x)$ (colored solid) corresponding to different noise levels. {\bf b}:~Estimation errors for the results in {\bf a}. }
\label{elsevier_cX_all}
\end{figure}
\begin{table} 
\caption{Relative errors of $\hat{c}(x)$ versus different noise levels.}
\label{error_cX}
\begin{center}
\begin{tabular}{cc} \hline
Noise Level &  Relative Error \\ \hline
0\%	&	    0.0001	\%	 \\
1\%	&	    0.2665	\%	\\
3\%	&	    0.8082	\%	\\
5\%	&	    1.3624	\%	\\
10\%	&	    2.8064	\%	\\ \hline
\end{tabular}
\end{center}
\end{table}
 \begin{figure} 
\centering
\begin{tabular}{cc}
\epsfig{file=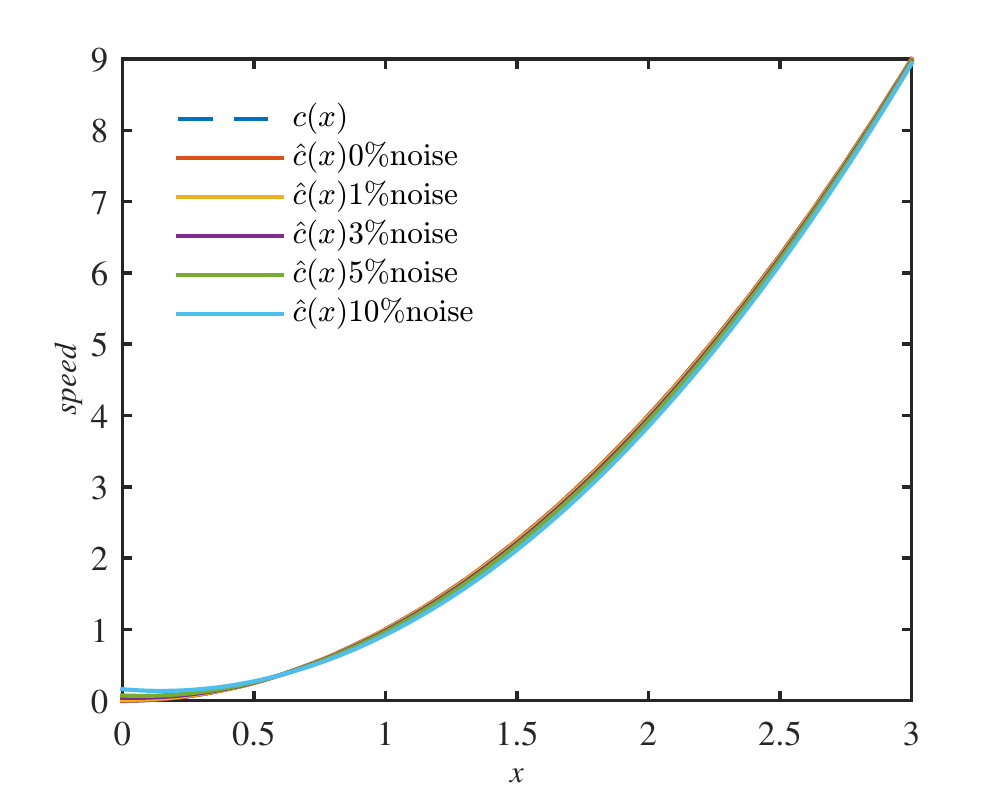,height=0.42\linewidth,width=0.5\linewidth,clip=} & \epsfig{file= 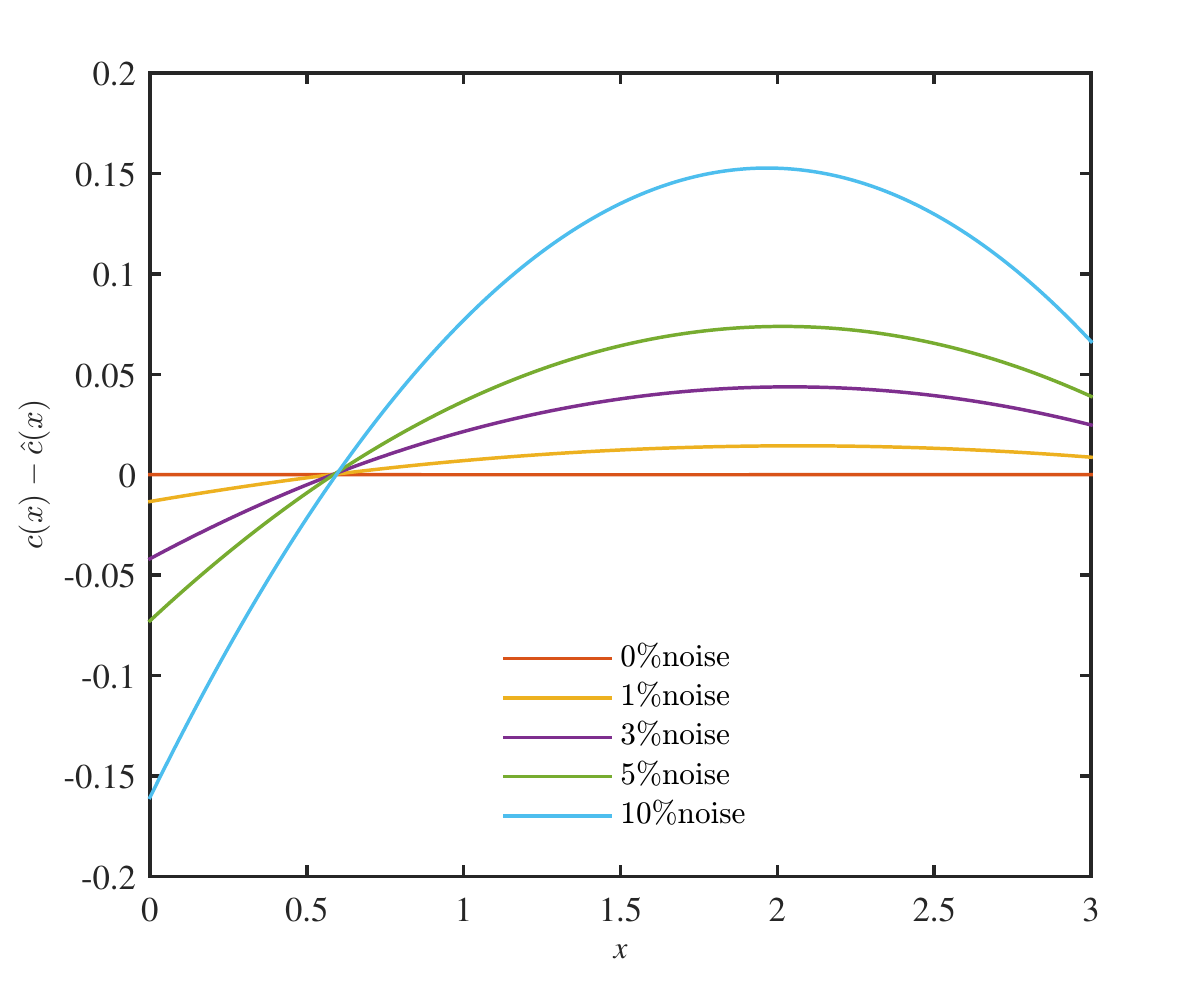, height=0.43\linewidth,width=0.5\linewidth,clip=}\\
 {\bf a} & {\bf b}
\end{tabular}
\caption{{\bf a}: Exact space-dependent velocity $c(x)$ (blue dashed) and  estimated ones $\hat{c}(x)$ (colored solid), in case of Hermite basis functions, corresponding to different noise levels. {\bf b}:~Estimation errors for the results in {\bf a}. }
\label{elsevier_cX_all_Hermite}
\end{figure}
\begin{table} 
\caption{Relative errors of $\hat{c}(x)$ versus different noise levels when Hermite basis functions have been used.}
\label{error_cX_Hermite}
\begin{center}
\begin{tabular}{cc} \hline
Noise Level &  Relative Error \\ \hline
0\%	&	0.0001	\%	\\
1\%	&	0.2665	\%	\\
3\%	&	0.8082	\%	\\
5\%	&	1.3624	\%	\\
10\%	&	2.8064	\%	\\ \hline
\end{tabular}
\end{center}
\end{table}
\begin{figure} 
\centering
\begin{tabular}{ccc}
\epsfig{file= 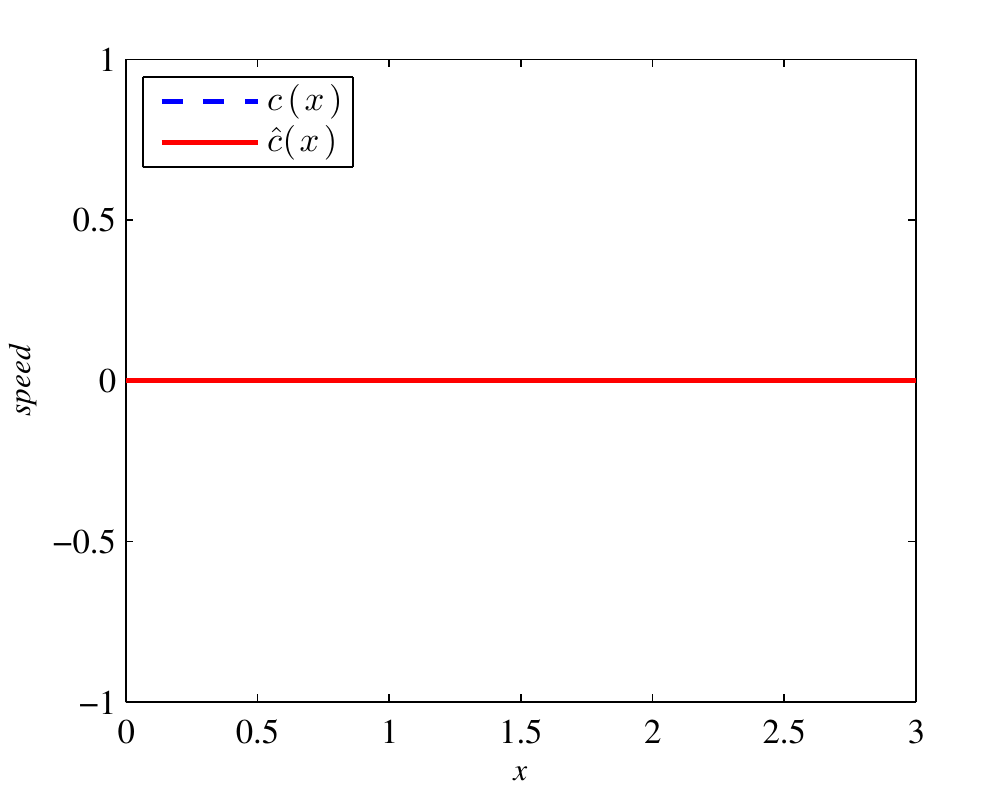,height=0.3\linewidth,width=0.32\linewidth,clip=} & \epsfig{file= 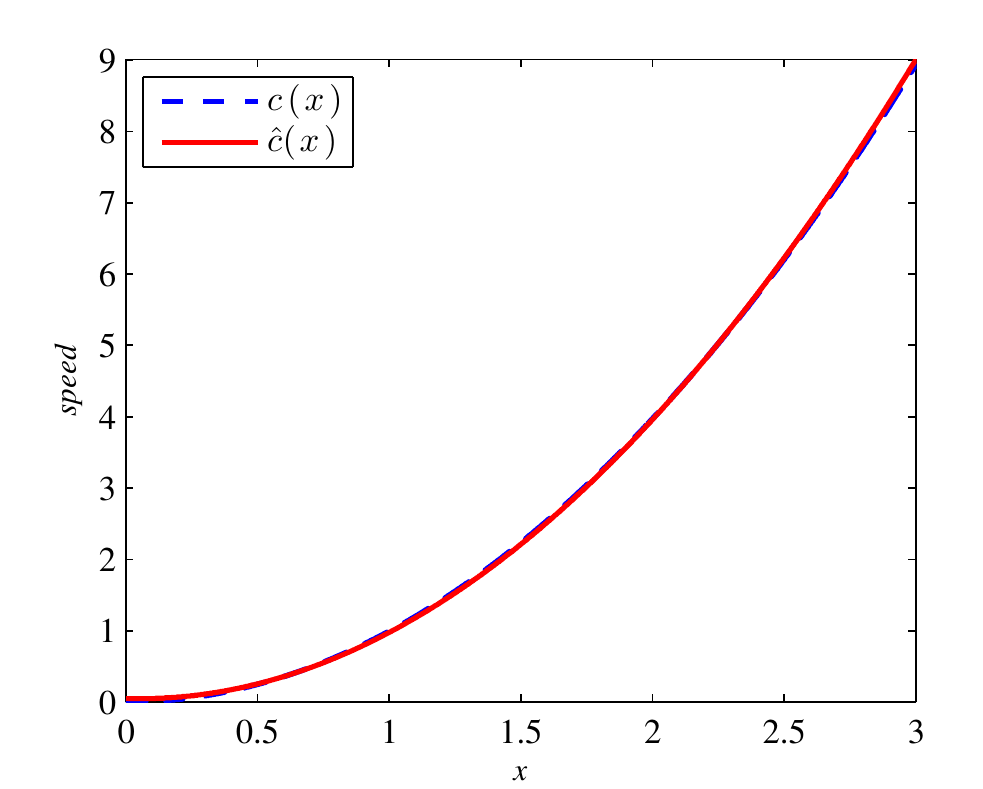,height=0.3\linewidth,width=0.32\linewidth,clip=} & \epsfig{file= 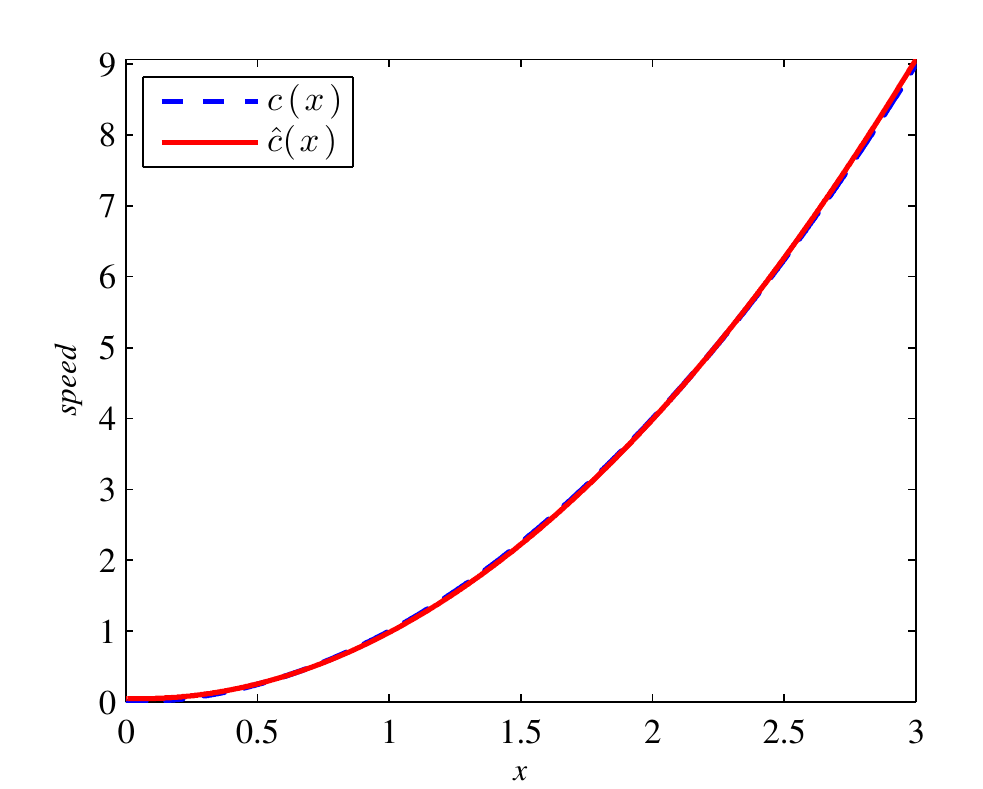,height=0.3\linewidth,width=0.32\linewidth,clip=}\\
$c(x,t^*_1)$ & $c(x,t^*_2)$ & $c(x,t^*_3)$
\end{tabular}
\caption{Exact velocity $c(x,t^*)$ (dashed blue) and estimated one $\hat{c}(x,t^*)$ (solid red) at three different fixed time.  The level of noise is $5\%$.}
\label{elsevier_cXT_noise5_3positions}
\end{figure}
\begin{figure} 
\centering
\begin{tabular}{cc}
\epsfig{file= 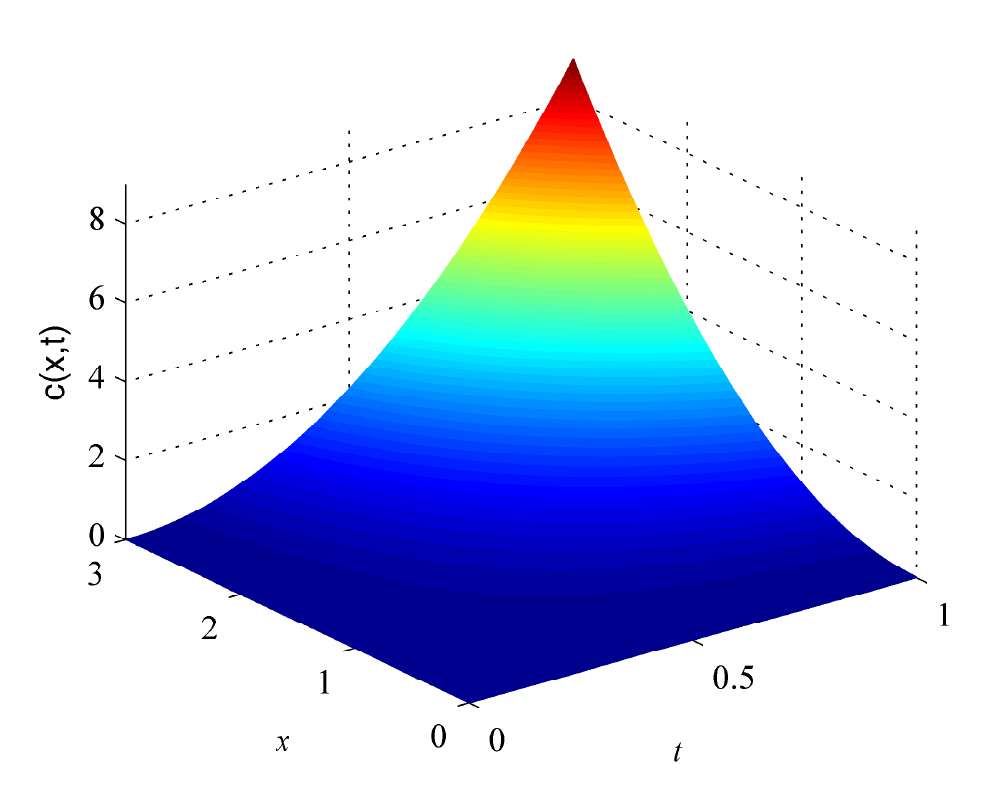,width=0.5\linewidth,clip=} & \epsfig{file= 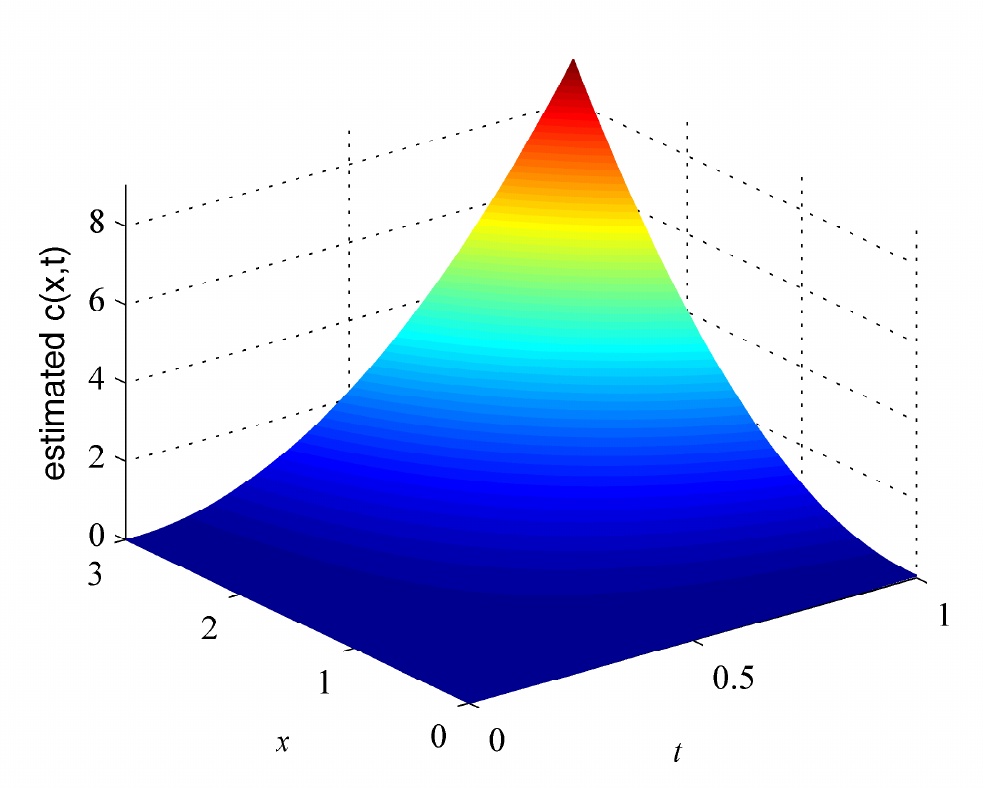,width=0.5\linewidth,clip=} 
\end{tabular}
\caption{Exact velocity $c(x,t)$ (left) and estimated one $\hat{c}(x,t)$ (right) after applying modulating functions-based algorithm and doing interpolation; noise level $= 5\%$.}
\label{elsevier_cXT_interp_noise5}
\end{figure}
\begin{table} 
\caption{Relative errors of $\hat{c}(x,t)$ versus different noise levels.}
\label{error_cXT}
\begin{center}
\begin{tabular}{cc} \hline
Noise Level &  Relative Error \\ \hline
0\%	&	0.0001\%	 \\
1\%	&	0.1983\%	\\
3\%	&	0.5989	\%	\\
5\%	&	1.0052	\%	\\
10\%&	2.0477	\%	\\ \hline
\end{tabular}
\end{center}
\end{table}
\subsubsection{{\rm IP3:}}
For the joint estimation, we set $f(x)=c(x)=x$. The estimated source and velocity in this joint estimation are shown in Figure~\ref{elsevier_joint_f_all} and Figure~\ref{elsevier_joint_cX_all}, respectively. The corresponding relative errors are presented in Table~\ref{error_joint}.
 \begin{figure} 
\centering
\begin{tabular}{cc}
 \epsfig{file=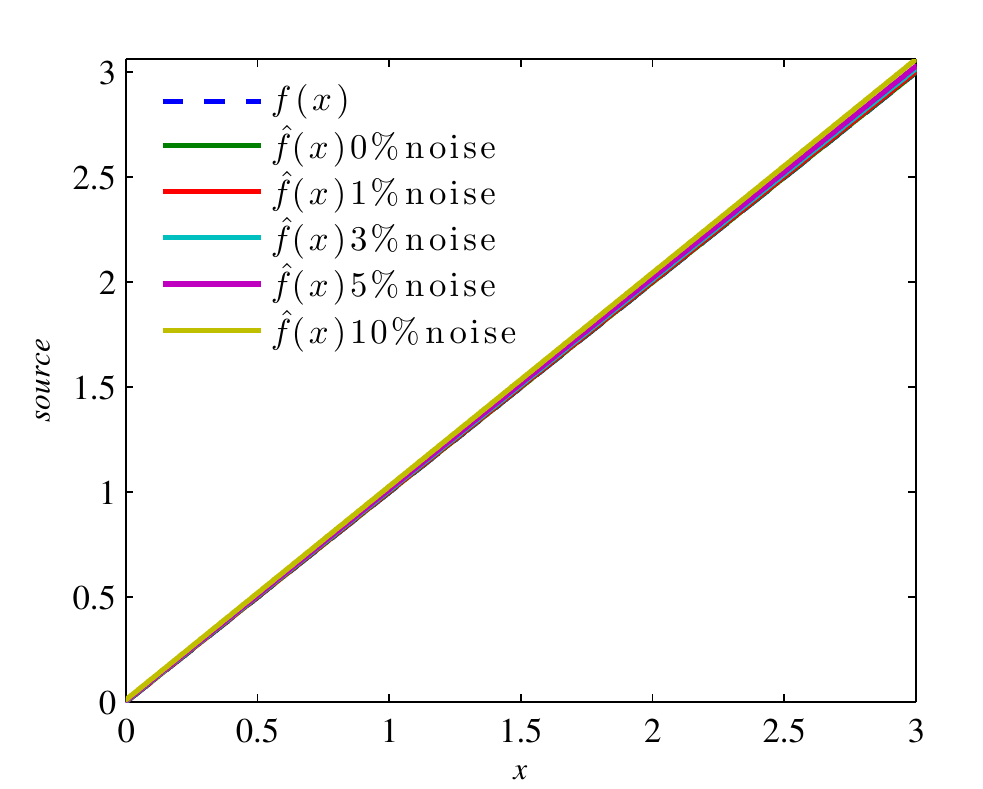,height=0.42\linewidth,width=0.5\linewidth,clip=} & \epsfig{file=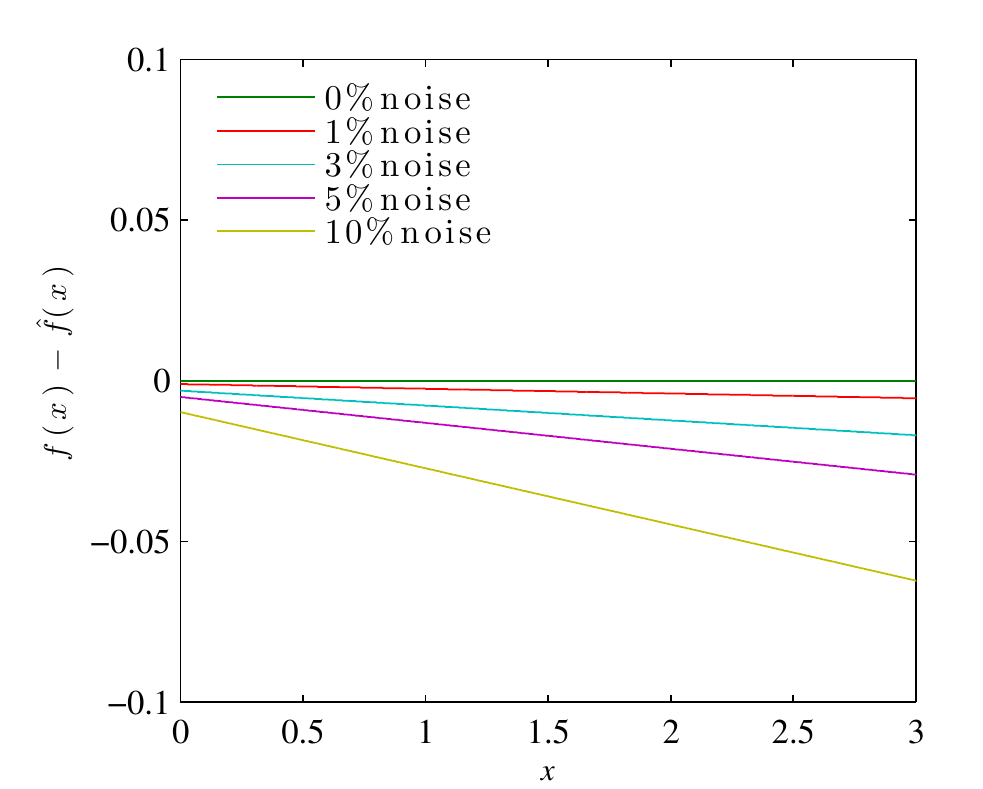,height=0.42\linewidth,width=0.52\linewidth,clip=} \\
 {\bf a} & {\bf b}
\end{tabular}
\caption{ {\bf a}: Exact source $f(x)$ (blue dashed) and  estimated ones in the joint estimation  $\hat{f}(x)$ (colored solid) corresponding to different noise levels. {\bf b}: Estimation errors for the results in {\bf a}.}
\label{elsevier_joint_f_all}
\end{figure}
  \begin{figure} 
\centering
\begin{tabular}{cc}
 \epsfig{file=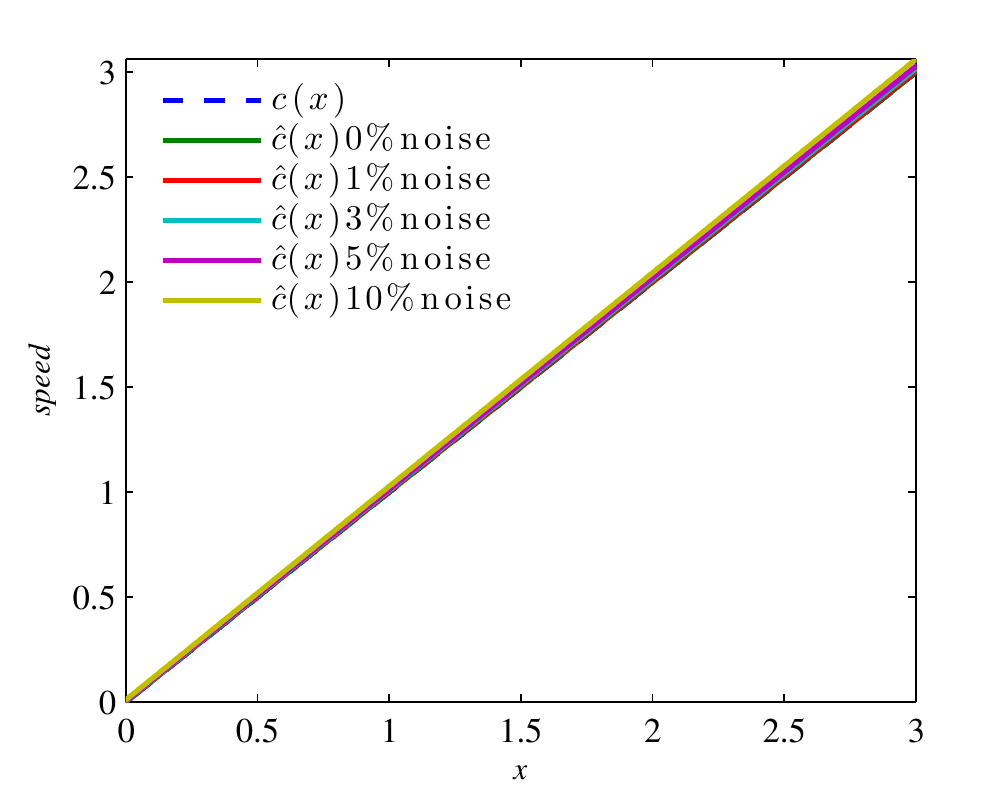,height=0.42\linewidth,width=0.5\linewidth,clip=} & \epsfig{file= 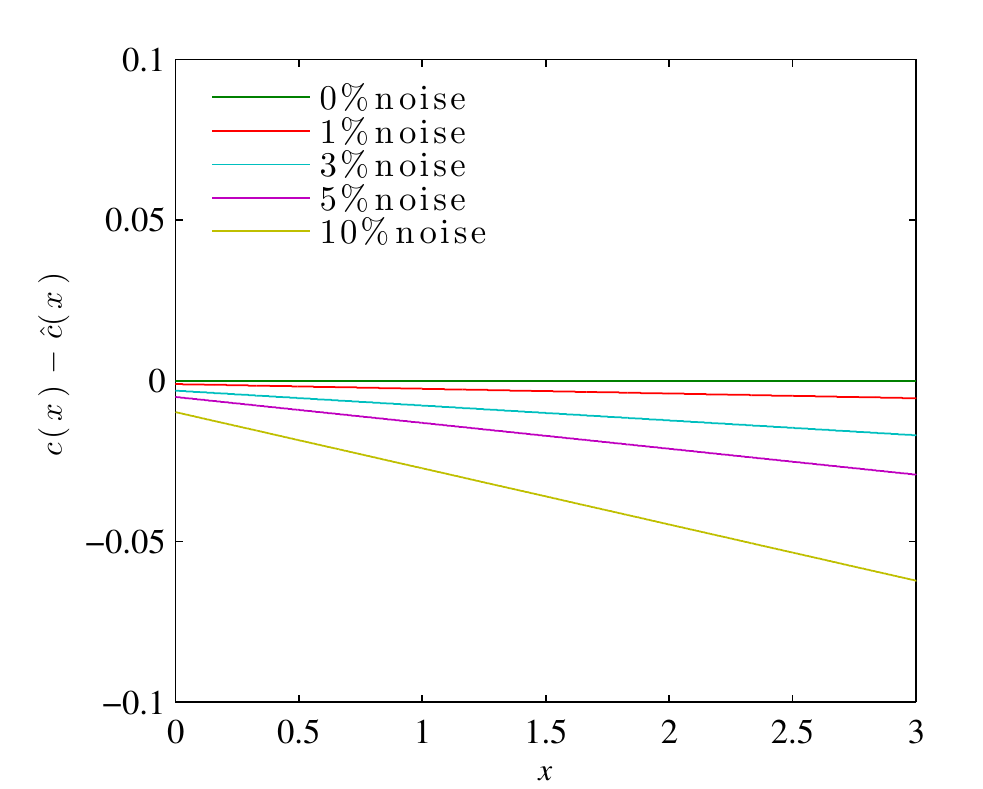,height=0.42\linewidth,width=0.52\linewidth,clip=}\\
 {\bf a} & {\bf b}
\end{tabular}
\caption{ {\bf a}: Exact  velocity $c(x)$ (blue dashed) and the estimated ones in the joint estimation $\hat{c}(x)$ (colored solid) corresponding to different noise levels. {\bf b}: Estimation errors for the results in {\bf a}.}
\label{elsevier_joint_cX_all}
\end{figure}
\begin{table} 
\caption{Relative errors of estimating $f(x)$ and $c(x)$ in the joint estimation versus different noise levels.}
\label{error_joint}
\begin{center}
\begin{tabular}{ccc } \hline
Noise Level &  Relative Error ($f$) & Relative Error ($c$) \\ \hline
0\%	&	    1.239e-05	\%	&	   5.4428e-05	\%	 \\
1\%	&	      0.20192	\%	&	      0.23044	\%	\\
3\%	&	      0.62443	\%	&	      0.68809	\%	\\
5\%	&	       1.0689	\%	&	       1.1411	\%	\\
10\%	&	       2.2539	\%	&	       2.2486	\%	\\ \hline
\end{tabular}
\end{center}
\end{table}
\par The presented figures and tables show that the estimated unknown is in quite good agreement with the exact one; therefore, the modulating functions-based method is an efficient and a robust method for solving parameters and source estimation for linear PDE. The method can be easily extended  to identify parameters of nonlinear PDE as shown in the next section.
\section{ Parameter estimation for  the 5th order KdV Equation}\label{sec_example2} 
\subsection{Method}
The 5th order nonlinear KdV equation has the following form:
\begin{equation}\label{kaw}
u_t(x,t) +\alpha_1 u(x,t) u_x(x,t) + \alpha_2 u_{xxx}(x,t) - \alpha_3 u_{xxxxx}(x,t)= 0,
\end{equation}
where $\alpha_1$, $\alpha_2$, and $\alpha_3$ are positive parameters. This equation is used to model different phenomena  such as pcapillaryÐgravity water waves, chains of coupled oscillators, and magneto-acoustic waves in plasma \cite{Be:77,GoOs:81,BrDe:02}. The parameters; $\alpha_1$, $\alpha_2$, and $\alpha_3$; are related to properties of the physical medium under consideration. Thus, estimating these parameters can be used to validate the applicability of this equation for particular media \cite{JaSe:14}. Hence, the following inverse problem can be defined:\\
\par  Given $u(x,t^*)$ and $u_t(x,t^*)$ at a fixed time $t^*$, find $\alpha_1$, $\alpha_2$, and $\alpha_3$.\\
\par The following proposition gives a solution to this inverse problem using modulating functions-based method.
 \begin{proposition}\label{IP_kaw}
Let $\alpha_1$, $\alpha_2$, and $\alpha_3$ be unknown parameters in (\ref{kaw}), and let $\{\phi_m(x)\}_{m=1}^{m=M}$ be a class of at least fifth order modulating functions  with $M \ge 3$. Then, the unknown parameters can be estimated by solving the following linear system: 
\begin{equation}\label{system_kaw}
 {\bf \mathcal{A}} { \Gamma} = {K},
\end{equation}
where the rows of the $M \times 3$ matrix $\mathcal{A}$, the elements of the vector $K$, and the vector $\Gamma$ are
\begin{equation}\label{A_kaw}
\mathcal{A}_{m} =
\left[
\begin{array}{c}
 -\frac{1}{2} \displaystyle{\int_0^L} u^2(x,t^*) \phi_m^{'}(x) \, \mathrm{d}x\\
  -\displaystyle{\int_0^L} u(x,t^*) \phi_m^{'''}(x) \, \mathrm{d}x \\
   \displaystyle{\int_0^L}u(x,t^*) \phi_m^{'''''}(x)  \, \mathrm{d}x
\end{array}
\right]^{tr},
\end{equation} 
\begin{equation}\label{K_kaw}
 K_m=  - \int_0^L  u_{t}(x,t^*) \phi_m(x) \, \mathrm{d}x,
\end{equation}
and
\begin{equation}\label{Gamma_kaw}
{ \Gamma}=
\left[
\begin{array}{ccc}
\alpha_1 &\alpha_2&\alpha_3
\end{array}
\right]^{tr},
\end{equation}
respectively. 
\end{proposition}
\begin{proof}
Appendix~\ref{app_ IP_kaw} provides details of the proof.
\end{proof}
\subsection{Numerical Simulations}\label{sec_numKDV}
For the numerical simulations, let $\alpha_1=\alpha_2=\alpha_3=1$, then (\ref{kaw}) is a Kawahara equation. Kawahara equation with  an initial condition $u(x,0)~= \frac{105}{169} \sech^4 \left[ \frac{1}{2\sqrt13} x \right]$ has the following exact solution \cite{Ka:03}:
\begin{equation}
u(x,t) = \frac{105}{169} \sech^4  \left[ \frac{1}{2\sqrt13} (x-\frac{36t}{169}) \right].
\end{equation}
Figure~\ref{parameters_Kawahara} and Table~\ref{kaw_table} exhibit the estimated parameters values and the relative errors, respectively; where $L=60$, $T=50$, $N_x=N_t=601$, and the degree of freedom $q$ in (\ref{poly_modu}) is chosen such that $\phi_m(x)$ is at least of order five. These results show that the identification of the parameters is successful in noise-free case. In the noisy case, the results are good for the lower-order coefficients, $\alpha_1,\alpha_2$, but less accurate for the coefficient of the higher order, $\alpha_3$, especially with a high level of noise as $10\%$. This result can be improved by adapting  the number of modulating function $M$ for this case. For example, if $M=8$, the relative error of estimating  $\alpha_3$ with $10\%$ of noise decreases to  $3.3897\%$.
 \begin{figure}
\centering
\epsfig{file= 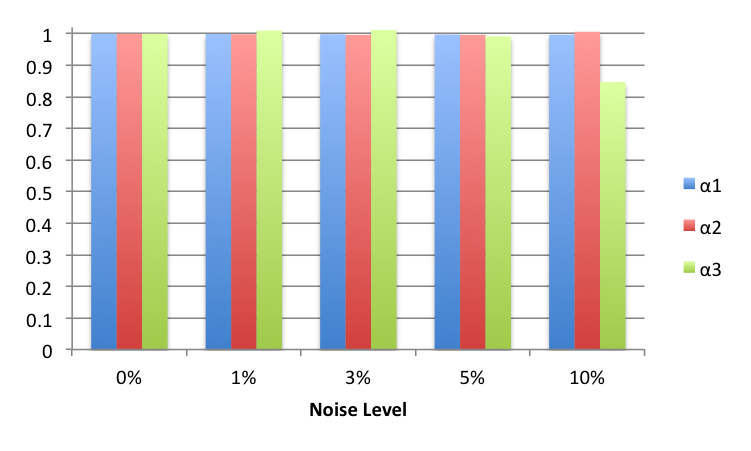,width=0.7\linewidth,clip=} 
\caption{Blue, red, and green bars represent the estimated  values of the  parameters  $\alpha_1, \alpha_2,$ and $\alpha_3$ in Kawarah equation, respectively. The exact values are  $\alpha_1=\alpha_2=\alpha_3=1$ and the estimation is done w.r.t different noise levels; $q=8$ and $M=9$.}
\label{parameters_Kawahara}
\end{figure}
\begin{table}
\caption{The relative errors of estimating the  parameters  $\alpha_1, \alpha_2,$ and $\alpha_3$ in Kawarah equation w.r.t different noise levels.}
\label{kaw_table}
\begin{center}
\begin{tabular}{cccc} \hline
Noise Level &  $\dfrac{|\alpha_1-\hat{\alpha}_1|}{|\alpha_1|}\times 100$ & $\dfrac{|\alpha_2-\hat{\alpha}_2|}{|\alpha_2|} \times 100$& $\dfrac{|\alpha_3-\hat{\alpha}_3|}{|\alpha_3|}\times 100$ \\ \hline
0\%	&	   2.8866e{-13}	\%	&	   4.2188e{-13}	\%	&	   2.377e{-11}	\%	 \\
1\%	&	     0.068971	\%	&	      0.18571	\%	&	      0.92843	\%	\\
3\%	&	      0.18305	\%	&	      0.39435	\%	&	       1.1323	\%	\\
5\%	&	      0.26548	\%	&	      0.38516	\%	&	      0.85491	\%	\\
10\%	&	      0.33414	\%	&	      0.59928	\%	&	       15.323	\%	\\ \hline
\end{tabular}
\end{center}
\end{table}
\par For constant unknowns, as in this example, it can be enough to have the measurements at a suitable subdomain $\omega=[0,L^*] \subset \Omega$. Estimating the three parameters in Kawahara equation with different values for $L^*$, $L^* \le L$, is shown in Figure~\ref{Kawahara_RError_Vs_L}. From this figure, we observe that in the noise-free case, the error is small even when the data is available only in the first third of the whole interval. In the noisy case, this subdomain should be increased to have an acceptable estimation error. In both cases, the error generally decreases as $L^*$ increases.
\begin{figure}
\centering
\begin{tabular}{cc}
\epsfig{file= 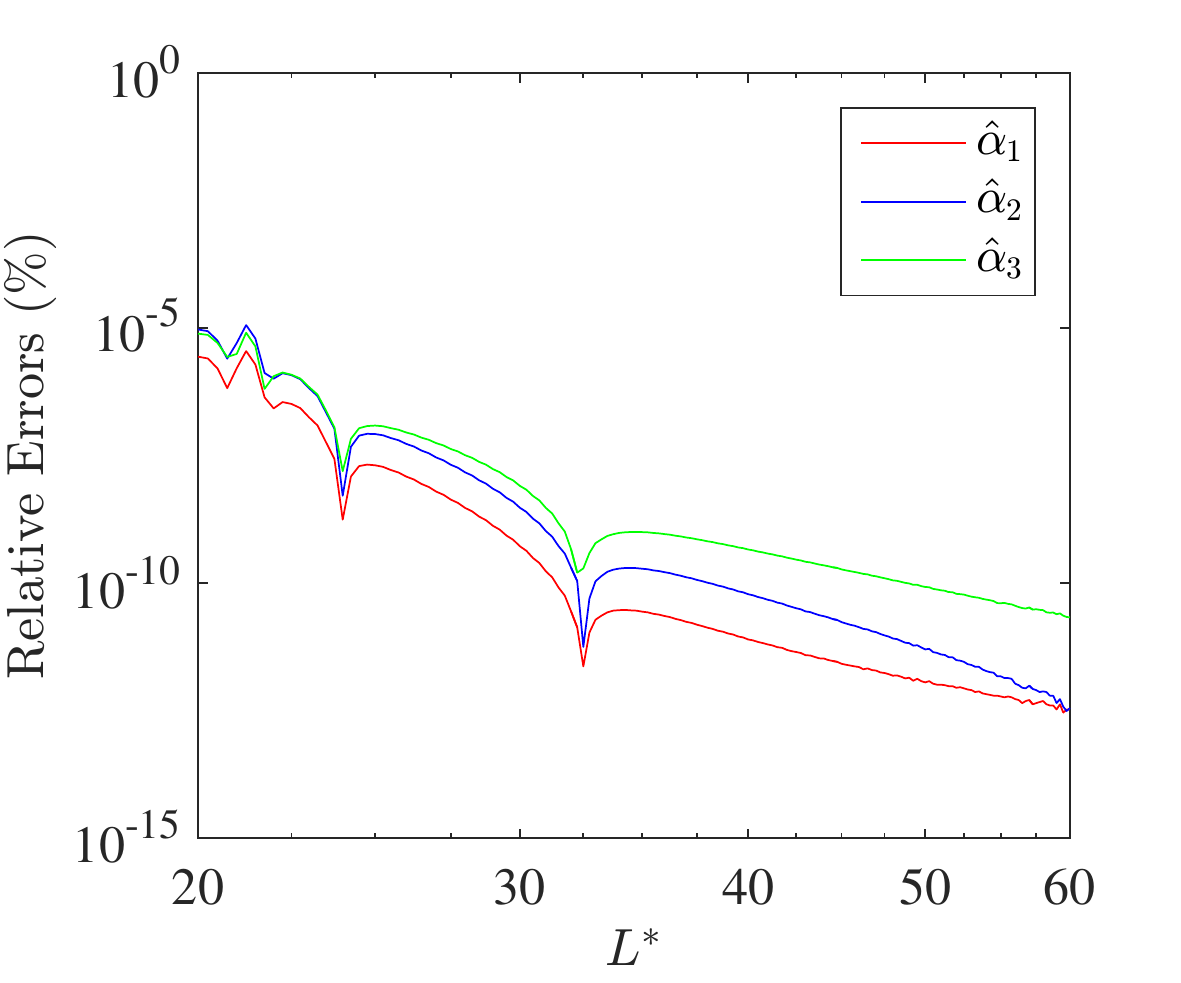,height=6cm,width=7cm,clip=} &
\epsfig{file= 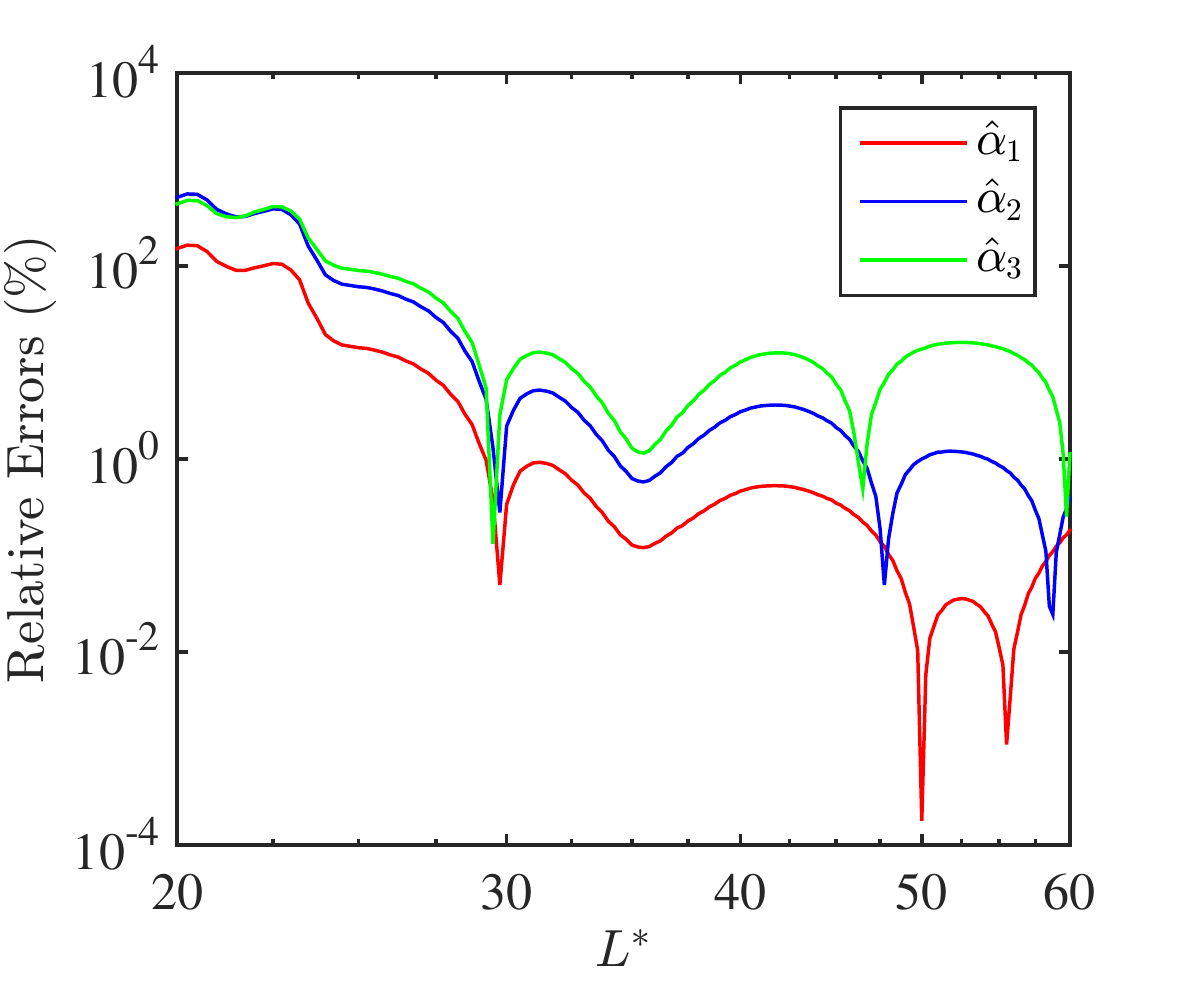,height=6cm,width=7cm,clip=}\\
 {\bf a} & {\bf b}
\end{tabular}
\caption{Relative Errors (in $\%$ and  log-log scale) for estimating $\alpha_1, \alpha_2,$ and $\alpha_3$ simultaneously in Kawarah equation w.r.t different $L^*$. {\bf a}: in noise-free case and {\bf b} with $3\%$ of noise on the measurements.}
\label{Kawahara_RError_Vs_L}
\end{figure}
\par Similar to the procedure developed for the wave equation example, one can estimate space-time dependent coefficients in the 5th order KdV equation. In addition, modulating functions-based method can be applied to estimate parameters in other high order nonlinear PDEs such as the sixth order Boussinesq equation and take advantage of  the properties of modulating functions to transfer the spatial derivatives to the modulating functions which can be computed analytically.
\section{Discussion}\label{sec_disc} 
The theoretical part in this paper confirms the efficiency of the modulating-functions based method and its simple implementation. In addition, the obtained  results  have shown the good performance of this method and its success even with  high levels of noise.
\par The number of  modulating functions, $M$, plays an important role in the performance of the method; (see Figure~\ref{optimal_M}). Figure~\ref{optimal_M} exhibits the number of modulating functions versus the relative error  for IP1, IP2, and IP3. Interestingly, it shows that the accuracy of the estimation can be improved by increasing $M$, especially in the noisy case.  In addition,  it shows that there exist  a unique optimal number of modulating functions, $M^*$, in the studied examples. However, the estimation is generally good for a relatively  large interval for $M$. Also, it proves that this optimal number depends on the considered problem. From this observation, it is worthy to know under  a specific  parameter identification problem and after choosing the type of modulating functions how the relative error, as a function of $M$, is affected by the noise level and the nature of the unknown functions. Figure~\ref{indp_table_v2} illustrates this behavior. The relative error, with respect to the number of modulating functions, is invariant with respect to the noise level and  the type of the unknown. This latter result offers a way to select an optimal number of modulating functions for real applications. In other words, one can set a synthetic function for the  unknown, apply the method and compute the error for different numbers of modulating functions to find the optimal one, $M^*$, and finally this $M^*$ might be a good guess for the optimal for the real inverse problem.
\begin{figure}[!t]
\centering
\begin{tabular}{cc}
{\bf a} & {\bf b} \\
 \epsfig{file= 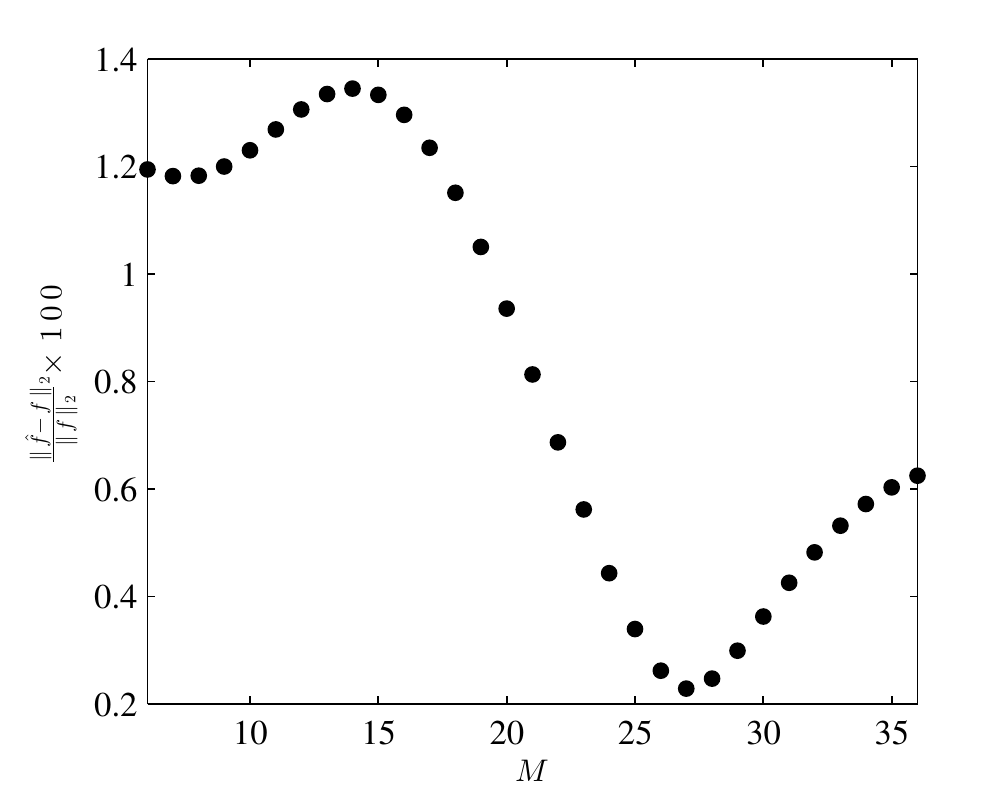,height=0.3\linewidth,width=0.45\linewidth,clip=} & \epsfig{file=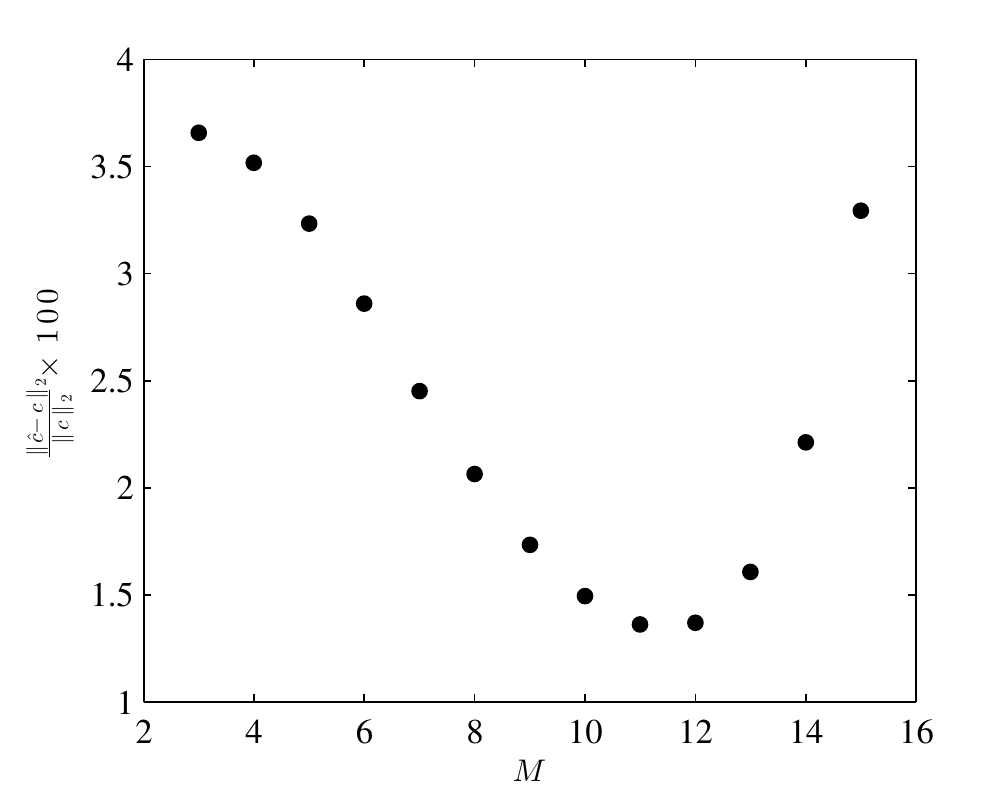,height=0.3\linewidth,width=0.45\linewidth,clip=} \\
{\bf  c} & {\bf d} \\
 \epsfig{file=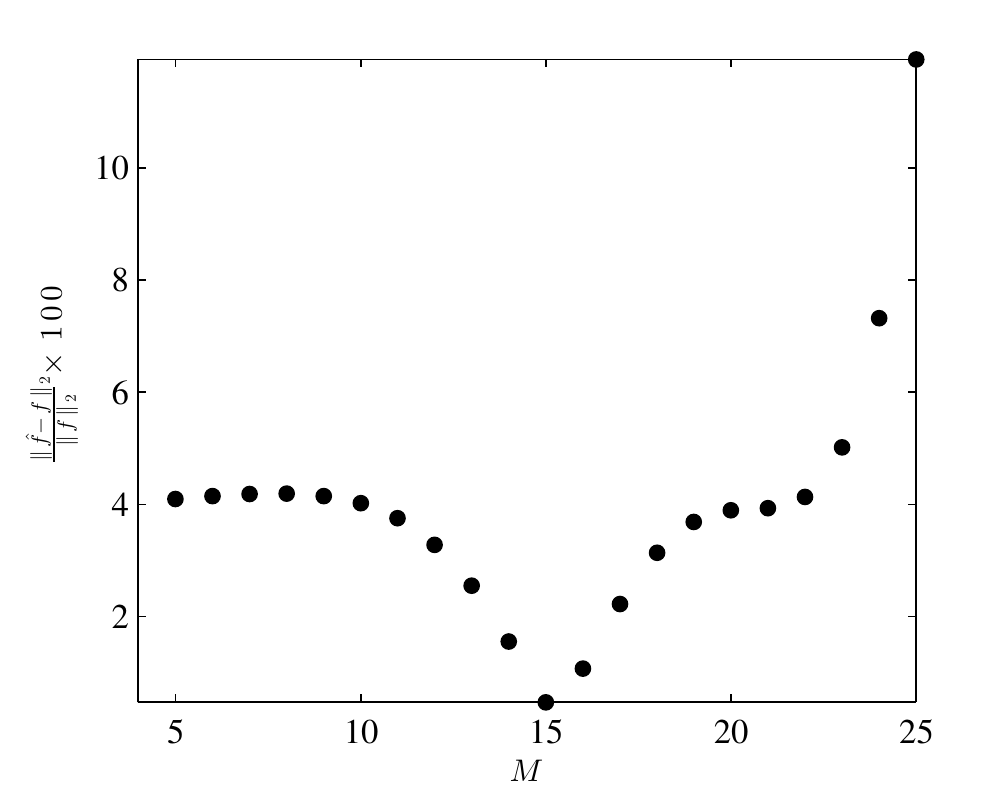,height=0.3\linewidth,width=0.45\linewidth,clip=}& \epsfig{file=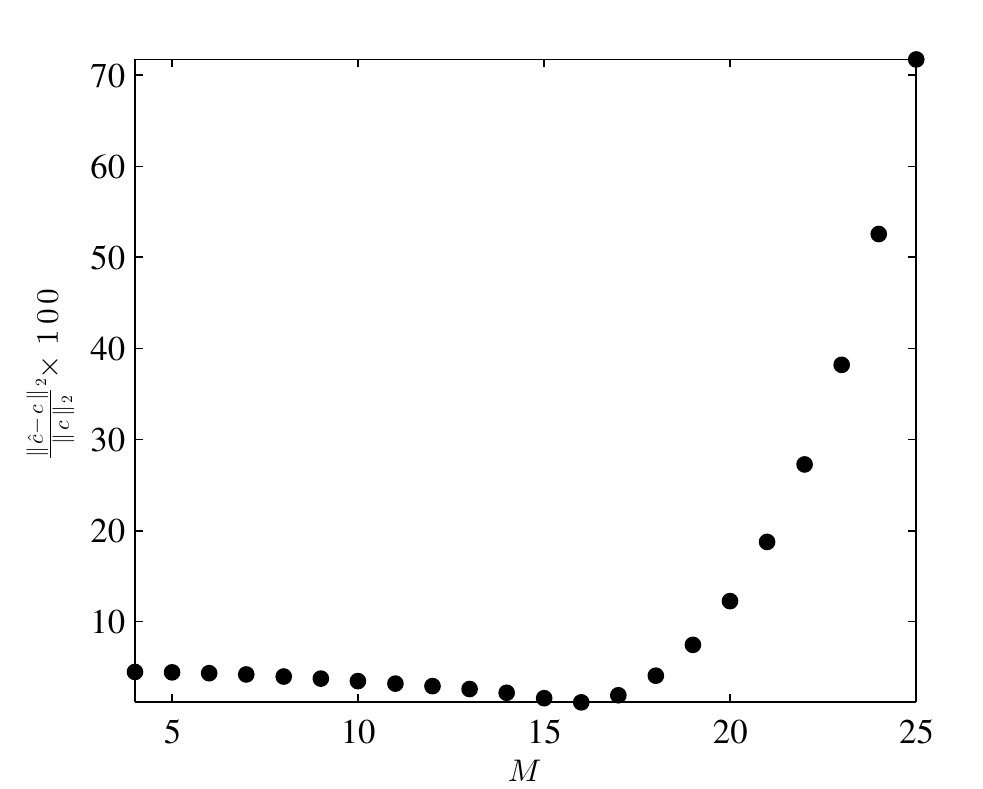,height=0.3\linewidth,width=0.45\linewidth,clip=} 
\end{tabular}
\caption{Number of modulating functions versus the relative errors for {\bf a}: IP1, {\bf b}: IP2, {\bf c}: IP3 source estimation, and  {\bf d}: IP3 velocity estimation. The optimal number of modulating functions, $M^*$, is $27$, $11$, $17$,  and $16$  in  {\bf a},  {\bf b},  {\bf c}, and  {\bf d}, respectively. In these figures, $5\%$ of the noise was added to the measurements.}
\label{optimal_M}
\end{figure}
\begin{figure}[!t]
\begin{center}
\begin{tabular}{p{0.8cm}cc} \hline
$f(x,t^*)$  & Polynomial functions & Sinusoidal functions \cite{Sh:57} \\ \hline
\rotatebox{90}{$\qquad \qquad \sin(x) {t^*}^2$} &
\epsfig{file= 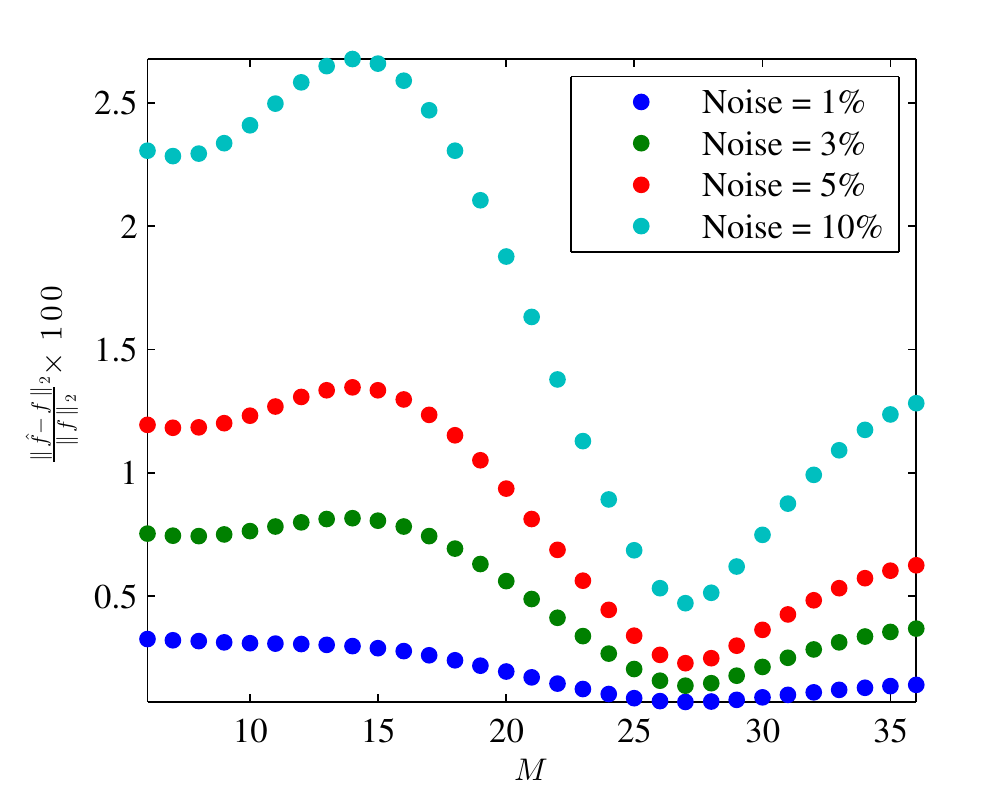,width=0.45\linewidth,clip=} & \epsfig{file= 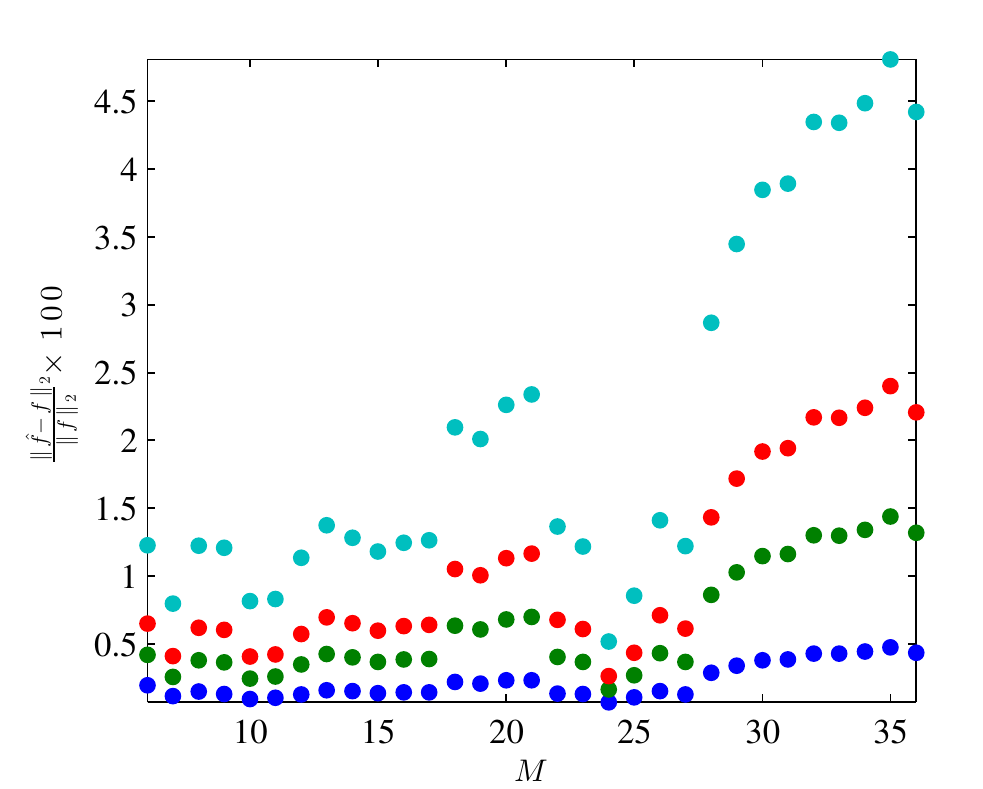,width=0.45\linewidth,clip=} \\
\rotatebox{90}{$\qquad \qquad \qquad  x^2 {t^*}^2$ }&
\epsfig{file= 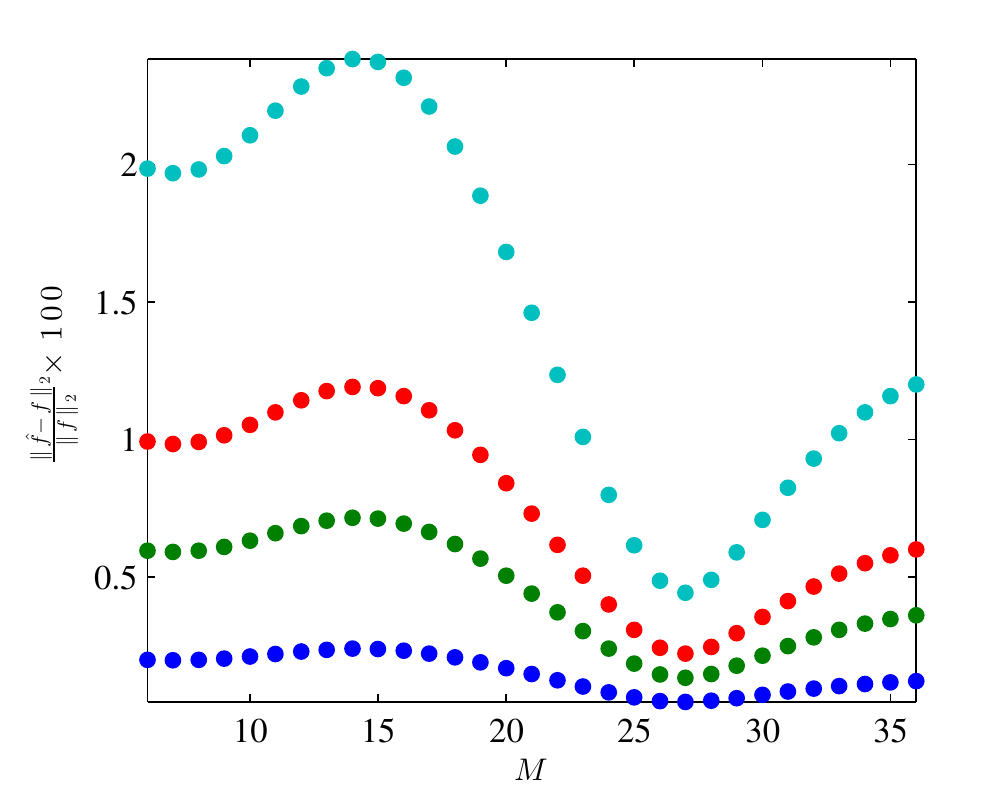,width=0.45\linewidth,clip=} & \epsfig{file= 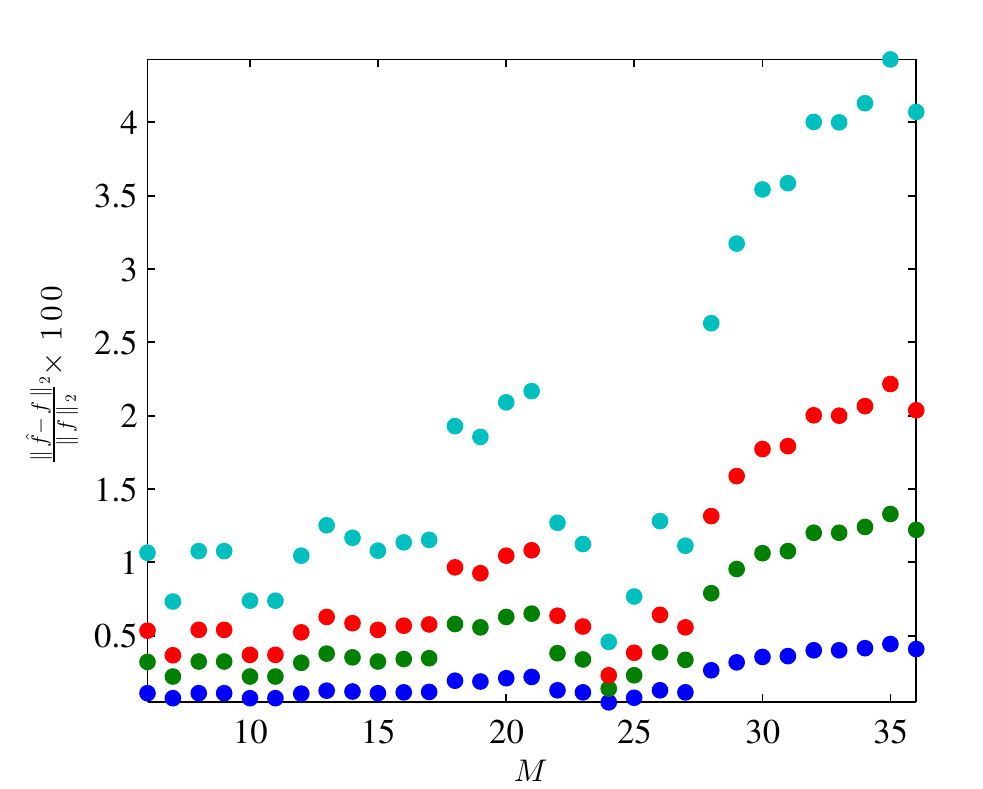,width=0.45\linewidth,clip=} \\ \hline
\end{tabular}
\end{center}
\caption{Number of modulating functions versus the relative error in {\bf IP1} w.r.t different  noise levels and unknown source. Two types  of modulating functions have been applied: polynomial modulating functions and sinusoidal modulating functions.}
\label{indp_table_v2}
\end{figure}
\par The choice of  an appropriate number of basis functions used to expand the space or time varying unknown functions is also important. Choosing this number large may lead to ill-conditioning issues. Moreover, if this number is significantly smaller than the appropriate one, we lose accuracy. Hence, this number must be selected  such that the numerical stability and accuracy are relatively good .
\par The approach can be also applied to the case of  measurements that are available at fixed points in the space instead of fixed time instants. However,  time-dependent modulating functions, $\phi(t)$, must be used in this case. 
\par In a forthcoming study, we will try to reduce the number of measurements and discuss the effect of this reduction on the three inverse problems considered in this paper.
\section{Conclusion} \label{sec_conc}
 In this paper, modulating functions-based method for solving inverse problems for 1D-PDEs has been proposed. The well-posedness of modulating functions-based solution have been studied. As illustrative examples, the method has been applied on the wave equation (linear 1D-PDE) and on the fifth order KdV (nonlinear 1D-PDE) to estimate different unknowns. By applying modulating functions-based method, the problem has been converted to a system of  algebraic equations which  is linear in the unknowns. Then these unknowns have been estimated using least square algorithms. Numerical simulations in both noise-free and noise-corrupted cases have shown good performance and robustness of this method.  The noise error contribution has been also studied and an upper bound has been derived and illustrated numerically.
 \par Future study will investigate the choice of the number of basis in order to propose an efficient and systematic method for selecting this number. In addition, extending modulating functions-based method to the estimation of  discontinuous space-time dependent unknowns, which are more realistic for real applications, will be studied. 
\section*{Funding} 
Research reported in this publication was supported by the King Abdullah University of Science and Technology (KAUST).
\section*{Appendices}
\appendix
\section{Proof of Proposition (\ref{IP12})}\label{app_IP12}
\subsection{IP1}
First,  at fixed time $t^*$,  the  PDE in (\ref{wave equation}) is multiplied by $\phi_m(x)$ and integrated over $[0,L]$:
\begin{equation}\label{step1_fxt}
 \int_0^L \phi_m(x) u_{tt}(x,t^*) \mathrm{d}x - c \int_0^L \phi_m(x) u_{xx}(x,t^*) \mathrm{d}x = \int_0^L \phi_m(x) f(x,t^*) \mathrm{d}x.
\end{equation}
Then we apply integration by parts twice to the second left-hand side integral in (\ref{step1_fxt}):
\begin{equation}\label{step2_fxt}
 \int_0^L \phi_m(x) u_{tt}(x,t^*) \mathrm{d}x - c \int_0^L \phi^{\prime\prime}_m(x) u(x,t^*) \mathrm{d}x = \int_0^L \phi_m(x) f(x,t^*) \mathrm{d}x.
\end{equation}
Finally, by writing $f(x,t^*)$ in its basis expansion, system (\ref{system_ct_compound}) is obtained with components as in (\ref{system_ct_compound_compoC}).IP1 and (\ref{system_ct_compound_compoK}).IP1.
\subsection{IP2} \label{app_IP2}
Equation (\ref{wave equation}) is first multiplied  by $\phi_m(x)$, where $c=c(x)$, and integrated over $\Omega$, we obtain:
\begin{equation}
\int_0^L \phi_m(x) u_{tt}(x,t^*) \mathrm{d}x -  \int_0^L \phi_m(x) c(x) u_{xx}(x,t^*) \mathrm{d}x = \int_0^L \phi_m(x) f(x) \mathrm{d}x.
\end{equation}
After that, we integrate the second term on the left-hand side by parts twice, and so
\begin{eqnarray*}
  \int_0^L u(x,t^*) \left[ c''(x) \phi_m (x)+2 c' (x)  \phi'_m(x)+  c (x) \phi_m'' (x)\right] \ \mathrm{d}x =\\
  \int_0^L \phi_m(x) u_{tt}(x,t^*) \mathrm{d}x -  \int_0^L \phi_m(x) f(x) \mathrm{d}x.
\end{eqnarray*}
Then by writting $c(x)$ in its basis expansion, system (\ref{system_ct_compound}) is obtained with components as in (\ref{system_ct_compound_compoC}).IP2 and (\ref{system_ct_compound_compoK}).IP2.
\section{Proof of Proposition (\ref{joint})}\label{app_joint}

System (\ref{system_joint}) can be obtained by doing the same steps in \ref{app_IP2}:  
Equation (\ref{wave equation}) is first multiplied  by $\phi_m(x)$; integrated over $\Omega$; then integration by parts is applied to the second integral, and finally,  the unknowns are written in their basis expansion, $f_I(x)~=~\sum_{i=0}^I \gamma_i \xi_i$ and $c_J(x)~=~\sum_{j=0}^J \beta_j \upsilon_j$.
\section{Proof of Proposition (\ref{IP_kaw})}\label{app_ IP_kaw}
\noindent {\bf STEP 1:} Fix the time in equation (\ref{kaw}) at $t^*$, and then  multiply the equation by the modulating functions $\phi_m(x)$:
\begin{equation}
\begin{split}
u_t(x,t^*)\phi_m(x) +\alpha_1 u(x,t^*) u_x(x,t^*)\phi_m(x) \\
+ \alpha_2 u_{xxx}(x,t^*) \phi_m(x)- \alpha_3 u_{xxxxx}(x,t^*) \phi_m(x)= 0.
\end{split}
\end{equation}
{\bf STEP 2:} 
Integrate over the space interval:
\begin{equation}\label{step2_kwa}
\begin{split}
\int_0^L u_t(x,t^*)\phi_m(x)\, \mathrm{d}x +\alpha_1\int_0^L  u(x,t^*) u_x(x,t^*)\phi_m(x) \, \mathrm{d}x\\
+ \alpha_2 \int_0^L u_{xxx}(x,t^*) \phi_m(x)\, \mathrm{d}x- \alpha_3 \int_0^L u_{xxxxx}(x,t^*) \phi_m(x)\, \mathrm{d}x= 0.
\end{split}
\end{equation}
{\bf STEP 3:}   By applying the integration by parts formula: once to the second integral, three times  to the third integral, and five times to the fourth integral in (\ref{step2_kwa}), one can obtain:
\begin{equation}\label{step3_kwa}
\begin{split}
\int_0^L u_t(x,t^*)\phi_m(x)\, \mathrm{d}x - \frac{1}{2} \int_0^L \alpha_1 u^2(x,t^*) \phi^\prime_m(x) \, \mathrm{d}x\\
-\int_0^L \alpha_2 u(x,t^*) \phi^{\prime\prime\prime}_m(x)\, \mathrm{d}x+ \int_0^L\alpha_3 u(x,t^*) \phi^{\prime\prime\prime\prime\prime}_m(x)\, \mathrm{d}x= 0.
\end{split}
\end{equation}
The first integral in (\ref{step3_kwa}) represents the $m^{th}$ row of $K$ as in (\ref{K_kaw}) while the second, third, and fourth integrals form the $m^{th}$ row of $\mathcal{A}$ multiplied by the vector of unknowns $\Gamma$, see (\ref{A_kaw}) and (\ref{Gamma_kaw}).


\end{document}